\documentclass[english,10pt]{amsart}
\usepackage[english]{babel}

\usepackage[matrix,arrow]{xy}
\xyoption{all}
\usepackage{amscd,amssymb,amsfonts,amsmath}

\usepackage{graphics}
\usepackage{epsfig}
\usepackage{mathrsfs}

\newcommand\mypagesizel{
\textwidth= 6.5in
\textheight=9in
\voffset-.55in
\hoffset -0.75in
\marginparwidth=56pt
}

\newcommand{\Pic}{\textup{Pic}}

\newcommand{\codim}{\textup{codim}}
\newcommand{\Sing}{\textup{Sing}}
\newcommand{\p}[0]{{\mathbb P}}

\newcommand{\Nef}{\textup{Nef}}

\newcommand{\Pef}{\textup{Psef}}

\newcommand{\NS}{\textup{N}^1}
\newcommand{\N}{\textup{N}}

\newcommand{\bNE}{\overline{\textup{NE}}}
\newcommand{\NE}{\textup{NE}}

\renewcommand{\phi}{\varphi}

\renewcommand{\le}{\leqslant}
\renewcommand{\ge}{\geqslant}

\newcommand{\sE}{\mathscr{E}}

\newcommand{\sI}{\mathscr{I}}

\newcommand{\sL}{\mathscr{L}}

\newcommand{\sO}{\mathscr{O}}

\mypagesizel

\newtheorem{thm}{Theorem}[section]
\newtheorem*{thm*}{Theorem}

\newtheorem{lemma}[thm]{Lemma}
\newtheorem{cor}[thm]{Corollary}

\newtheorem{prop}[thm]{Proposition}

\theoremstyle{definition}

\newtheorem{say}[thm]{}
\newtheorem{exmp}[thm]{Example}

\newtheorem{defn-thm}[thm]{Definition-Theorem}

\newtheorem{rem}[thm]{Remark}

\theoremstyle{remark}

\newtheorem*{not-and-def}{Notation and definitions}

\numberwithin{equation}{section}

\begin{document}

\title[On Fano varieties whose effective divisors are nef]{On Fano varieties whose effective divisors are numerically eventually free}

\author{St\'ephane \textsc{Druel}}

\address{St\'ephane Druel: Institut Fourier, UMR 5582 du
  CNRS, Universit\'e Grenoble 1, BP 74, 38402 Saint Martin
  d'H\`eres, France} 

\email{Stephane.Druel@ujf-grenoble.fr}


\begin{abstract}

In this paper we classify mildly singular Fano varieties 
with maximal Picard number
whose effective divisors are numerically eventually free.
\end{abstract}

\maketitle

\tableofcontents

\section{Introduction}

Let $X$ be a normal projective variety and consider the finite dimensional dual $\mathbb{R}$-vector spaces 
$$\N_1(X)=\big(\{1-\text{cycles}\}/\equiv\big)\otimes\mathbb{R} \quad \text{and} \quad \NS(X)=\big(\Pic(X)/\equiv\big)\otimes\mathbb{R},$$
where $\equiv$ denotes numerical equivalence. The dimension $\rho(X)=\dim \NS(X)=\dim \N_1(X)$ is called the 
\emph{Picard Number} of $X$. The \emph{Mori cone} of $X$ is the closure $\bNE(X)\subset \N_1(X)$ of the cone 
$\NE(X)$
spanned by classes of effective curves. Its dual cone is the \emph{nef cone} 
$\Nef(X)\subset\NS(X)$, which by Kleiman's criterion is the closure of the cone spanned by ample classes.
The closure of the cone spanned by effective classes  in $\NS(X)$ is the pseudo-effective cone $\Pef(X)$.
These cones
$$\Nef(X)\subset \Pef(X)\subset \NS(X)$$
carry geometric information about the variety $X$.
It is natural to try to describe (normal projective) varieties $X$ with $\Nef(X) = \Pef(X)$.
The simplest examples of complex projective manifolds with $\Nef(X) = \Pef(X)$ are given by 
manifolds with Picard number $1$ and
homogeneous spaces.
The case of projective space bundles over curves was worked out by Fulger in \cite{fulger}.
If $\sE$ is a locally free sheaf of finite rank on a smooth complex projective curve $C$, then $\Nef\big(\mathbb{P}_C(\sE)\big)=\Pef\big(\mathbb{P}_C(\sE)\big)$ if and only if $\sE$ is semistable
(see \cite[Lemma 3.2]{fulger}). 
A smooth projective toric variety $X$ satisfies $\Nef(X) = \Pef(X)$ if and only if $X$ is isomorphic to a product of projective spaces by \cite[Proposition 5.3]{fujino_sato}. The case of
smooth projective horospherical varieties is addressed in \cite{qifeng}.

Let $X$ be a normal projective variety such that $K_X$ is $\mathbb{Q}$-Cartier. 
We say that $X$ is \emph{$\mathbb{Q}$-Fano} if $-K_X$ is ample.
The Mori cone of a $\mathbb{Q}$-Fano variety with log canonical singularities is rational polyhedral, and generated by classes of curves (see \cite[Theorem 16.6]{fujino_mmp}). The geometry of
$X$ is reflected to a large extent in the combinatorial properties of $\NE(X)=\bNE(X)$. Every face $V$ of 
$\NE(X)$ corresponds to a surjective morphism with connected fibers $\phi\colon X \to Y$ onto a normal projective variety, wich is called a \emph{Mori contraction}. The morphism $\phi$ contracts precisely those curves on $X$ with class in $V$ (see \cite[Theorem 16.4]{fujino_mmp}). Conversely, any morphism with connected fibers onto a normal projective
variety arises in this way.

In this paper we address mildly singular $\mathbb{Q}$-Fano varieties with 
$\Nef(X)=\Pef(X)$. 

\begin{thm}\label{thm:main}
Let $X$ be a $\mathbb{Q}$-Fano variety
with locally factorial canonical singularities. Suppose that
$\Nef(X)=\Pef(X)$. Then $\rho(X) \le \dim X$, and equality holds only if    
$X \cong X_1 \times \cdots \times X_m$
where $X_i$ is a double cover of 
$\p^1\times\cdots\times\p^1$ branched along a reduced divisor of type
$(2,\ldots,2)$ for each 
$i \in \{1,\ldots,m\}$.
\end{thm}

\begin{rem}Example \ref{exmp:quotient_singularities} shows that the statement
Theorem \ref{thm:main} does not hold for $\mathbb{Q}$-Fano varieties with
Gorenstein canonical singularities. 
\end{rem}

Occhetta, Sol\'a Conde, Watanabe, and Wi\'sniewski
recently posted a
somewhat related result. They proved in \cite{OSCWW} that
a Fano manifold whose elementary contractions are $\p^1$-fibrations is rational homogenous, without any assumption on the Picard number of $X$ (see also \cite[Proposition 2.4]{bcdd03}).

\medskip

The argument for the proof of Theorem \ref{thm:main} goes as follows. 
Suppose furthermore that $\dim X \ge 3$.
The first step in the proof of Theorem 
\ref{thm:main} is to show that the Mori cone $\NE(X)$ is simplicial. This is done in Lemma \ref{lemma:nef_pseff_picard}.
Then we argue by induction on the dimension of $X$. If $X \to Y$ is a contraction, then it is easy to see that $Y$ satisfies all the conditions listed in Theorem \ref{thm:main}. 

Suppose first that there is a contraction $X \to Y$ where $Y$ is a double cover of $\p^1\times \cdots \times \p^1$ branched along a divisor of type $(2,\ldots,2)$. It is easy to see that we must have $\dim Y \ge 3$ (see Proposition \ref{proposition:nef_psef_surfaces}) and that there is a contraction $X \to T$ such that the induced morphism 
$X \to Y \times T$ is surjective and finite. 
We conclude that $X \cong Y \times T$ using the following result.

\begin{thm}\label{thm_intro:finite_onto_double_cover}
Let $Z$ be a $\mathbb{Q}$-Fano variety of dimension $\ge 3$ with Gorenstein canonical singularities, $Y$ a double cover of $\p^1\times \cdots \times \p^1$ branched along a reduced divisor $B=\sum_{j\in J}B_j$ of type
$(2,\ldots,2)$, and $f\colon Z \to Y$ a finite morphism. Suppose that 
$B_j$ is ample for each $j\in J$.
Then $\deg(f)=1$.
\end{thm}

Suppose now that for each contraction
$X \to Y$, we have $Y\cong \p^1\times \cdots \times \p^1$. Then apply the following characterization result to complete the proof of Theorem \ref{thm:main}.

\begin{prop}\label{proposition:intro}
Let $X$ be a $\mathbb{Q}$-Fano variety of dimension $\ge 3$ with Gorenstein canonical singularities. 
Suppose that $\Nef(X)=\Pef(X)$ and $\rho(X)=\dim X$. Suppose furthermore
that each non-trivial contraction $X \to Y$ with $\dim Y < \dim X$ satisfies $Y\cong \p^1\times \cdots \times \p^1$.
Then, either $X\cong \p^1\times \cdots \times \p^1$, or $X$ is a double cover of 
$\p^1\times\cdots\times\p^1$ branched along a reduced divisor of type $(2,\ldots,2)$.   
\end{prop}

\medskip

In order to prove Theorem \ref{thm_intro:finite_onto_double_cover}, we are led to study finite morphisms between some Del Pezzo surfaces (see Propositions \ref{proposition:finite_morphism_quadric_bis} and \ref{proposition:double_cover_ter}). In section \ref{section:finite_morphisms_surfaces}, we address finite morphisms between smooth Del Pezzo surfaces. Beauville classified in \cite{beauville_endo} the smooth Del Pezzo surfaces which admit an endomorphism of degree $>1$. 

\begin{prop}[{\cite[Proposition 3]{beauville_endo}}]\label{prop:endo_surface}
A smooth Del Pezzo surface $S$ admits an endomorphism of degree $>1$ if and only if $K_S^2 \ge 6$.
\end{prop}

Our result is the following.

\begin{thm}\label{thm:morphism_Del_pezzo_surface}
Let $S$ and $T$ be smooth Del Pezzo surfaces with $K_S^2<K_T^2$, and let 
$f \colon S \to T$ be a finite morphism. 
Then
$K_T^2 \ge 8$.
\end{thm}

\

\noindent {\bf Acknowledgements.} We are grateful to Cinzia \textsc{Casagrande} for very fruitful discussions.

\section{Notation and conventions}

Throughout this paper we work over the field of complex numbers. Varieties are always assumed to be reduced and irreducible. 

We denote by $\Sing(X)$ the singular locus of a variety $X$.

Let $X$ be a normal projective variety, and
$B=\sum a_i B_i$ an effective $\mathbb{Q}$-divisor on $X$, i.e., $B$ is  a nonnegative $\mathbb{Q}$-linear combination 
of distinct prime Weil divisors $B_i$'s on $X$. 
Suppose that $K_X+B$ is $\mathbb{Q}$-Cartier, i.e.,  some nonzero multiple of it is a Cartier divisor. 
Let $\mu:\tilde X\to X$ be a log resolution of the pair $(X,B)$. 
This means that $\tilde X$ is a smooth projective
variety, $\mu$ is a birational projective morphism whose exceptional locus is the union of prime divisors $E_i$'s, 
and the divisor $\sum E_i+\tilde{B}$ has simple normal crossing 
support, where $\tilde{B}$ denotes the strict transform of $B$ in $X$.  
There are uniquely defined rational numbers $a(E_i,X,B)$'s such that
$$
K_{\tilde X}+\tilde{B} = \mu^*(K_X+B)+\sum a(E_i,X,B)E_i.
$$
The $a(E_i,X,B)$'s do not depend on the log resolution $\mu$,
but only on the valuations associated to the $E_i$'s. 
We say that $(X,B)$ is \emph{canonical} (respectively, \emph{log terminal} or \emph{klt}) if  all $a_i\le 1$ (respectively, $a_i<1$), and, for some  log resolution 
$\mu:\tilde X\to X$ of $(X,B)$, $a(E_i,X,B)\ge 0$ (respectively, $a(E_i,X,B)\ge -1)$
for every $\mu$-exceptional prime divisor $E_i$.
We say that $(X,B)$ is \emph{log canonical} if, for some  log resolution 
$\mu:\tilde X\to X$ of $(X,B)$, $a(E_i,X,B) \ge -1$ 
for every $\mu$-exceptional prime divisor $E_i$.
If these conditions hold for some log resolution of $(X,B)$, then they hold for every  
log resolution of $(X,B)$.
We say that $X$ is canonical (respectively  log canonical)  if so is $(X,0)$.
We say that $X$ is \emph{Gorenstein}, if $X$ is locally Cohen-Macaulay and $K_X$ is Cartier.
Note that if $X$ is Gorenstein, then $X$ is canonical if and only if $X$ is klt.
We say that $X$ is $\mathbb{Q}$-\emph{Gorenstein}, if $K_X$ is $\mathbb{Q}$-Cartier.

Let $X$ be a normal projective variety such that $K_X$ is $\mathbb{Q}$-Cartier. 
We say that $X$ is \emph{$\mathbb{Q}$-Fano} if $-K_X$ is ample.
We say that a normal surface $S$ is a \emph{Del Pezzo} surface if $S$ is $\mathbb{Q}$-Fano with canonical singularitites. Note that a normal surface $S$ is canonical if and only if $S$ has Du Val singularities (see \cite[Theorem 4.5]{kollar_mori}).

Let $X$ be a normal projective variety, and
$B$ an effective $\mathbb{Q}$-divisor on $X$ such that $K_X+B$ is $\mathbb{Q}$-Cartier.
We say that $(X,B)$ is \emph{$\mathbb{Q}$-Fano} if $-(K_X+B)$ is ample. 

\medskip

If $\sE$ is a locally free sheaf of $\sO_X$-modules on a variety $X$, 
we denote by $\p_X(\sE)$ the Grothendieck projectivization $\textup{Proj}_X(\textup{Sym}(\sE))$,
and by $\sO_{\p_X(\sE)}(1)$ its tautological line bundle.

Given a positive integer $m$, we denote by $\mathbb{F}_m$ the surface $\mathbb{P}_{\p^1}\big(\sO_{\p^1}\oplus\sO_{\p^1}(-m)\big)$.

Given line bundles $\sL_1$ and $\sL_2$ on varieties $X_1$ and $X_2$, we denote by $\sL_1\boxtimes\sL_2$ the line bundle $\pi_1^*\sL_1\otimes\pi_2^*\sL_2$ on $X_1\times X_2$, where $\pi_1$ and $\pi_2$ are the projections onto 
$X_1$ and $X_2$, respectively.

\section{Double covers}

In this section we gather some properties of double covers of smooth (projective) varieties.

\begin{say}\label{say:double_cover}
Let $f \colon X \to Y$ be a finite surjective morphism of degree $2$. Suppose that $X$ is Cohen-Macaulay and $Y$ is smooth. Then  
there exist a line bundle $\sL$ on $Y$, a section $s\in H^0(Y,\sL^{\otimes 2})$, and an isomorphism $f_*\sO_X\cong \sO_Y\oplus \sL^{\otimes -1}$ of $\sO_Y$-algebras,
where the structure of $\sO_Y$-algebra on 
$\sO_Y\oplus \sL^{\otimes -1}$ is induced by $s^\vee\colon \sL^{\otimes -2} \to \sO_Y$. 
We refer to \cite{cossec_dolgachev} for details.
This implies that $X$ is Gorenstein with dualizing sheaf $\omega_X\cong f^*(\omega_Y\otimes \sL)$. 
If moreover $Y$ is projective, then $X$ is $\mathbb{Q}$-Fano if and only if $(\omega_Y\otimes \sL)^{\otimes -1}$ is an ample line bundle.

Denote by $B$ the divisor of zeroes of $s$. A straightforward local computation shows that $X$ is normal if and only if $B$ is reduced. By \cite[Proposition 5.20]{kollar_mori}, $X$ is canonical if and only if $(Y,\frac{1}{2}B)$ is klt.
\end{say}

\begin{lemma}\label{lemma:degree_double_cover}
Fix an integer $n \ge 1$, and let $f \colon X \to (\p^1)^n$ be a double cover branched along a reduced divisor $B$. 
Then $X$ is $\mathbb{Q}$-Fano if and only if $B$ has type
$(2d_1,\ldots,2d_n)$ with $d_i \in\{0,1\}$ for each $i\in\{1,\ldots,n\}$.
\end{lemma}

\begin{rem}In the setup of Lemma \ref{lemma:degree_double_cover}, denote by $m$ the cardinality of 
the set $\{1 \le i \le n\,|\,d_i=0\}$. Then 
$X \cong (\p^1)^m \times Y$ where $Y$ is a 
double cover of 
$(\p^1)^{n-m}$ branched along a 
reduced divisor of type $(2,\ldots,2)$.
\end{rem}

\begin{lemma}\label{lemma:picard_number_double_cover_bis}
Let $X$ and $Y$ be projective Fano manifolds, and 
let $X \to Y$ be a double cover branched along an ample divisor. Suppose furthermore that $\dim X=\dim Y \ge 3$.
Then $\rho(X)=\rho(Y)$. 
\end{lemma}

\begin{proof}
From \ref{say:double_cover} and \cite[Theorem 2.1]{lazarsfeld_barth}, we conclude that $b_2(X)=b_2(Y)$. 
Since $X$ and $Y$ are Fano manifolds, we also have $\rho(X)=b_2(X)$ and $\rho(Y)=b_2(Y)$,
proving the lemma.
\end{proof}

\begin{cor}\label{cor:picard_number_double_cover}
Let $X$ be a smooth double cover of $\p^1\times\cdots\times\p^1$ branched along a reduced divisor $B$ of type
$(2,\ldots,2)$ with $\dim X \ge 3$. Then $\rho(X)=\dim X$. 
\end{cor}

\begin{proof}
This follows from Lemma \ref{lemma:picard_number_double_cover_bis} together with Lemma \ref{lemma:degree_double_cover}.
\end{proof}

The following example shows that the statement of Lemma \ref{lemma:picard_number_double_cover_bis} does not hold for surfaces.

\begin{exmp}Let $S$ be a double cover of
$\p^1\times\p^1$ branched along a smooth divisor of type $(2,2)$. Then $S$ is a (smooth) Del Pezzo 
surface of degree $K_S^2=4$, and $\rho(S)=6$.
\end{exmp}

We will need the following observations.

\begin{lemma}\label{lemma:factoriality_hypersurface_ring}
Let $(R,\mathfrak{m})$ be a regular local ring, and let $f_1,f_2\in\mathfrak{m}$ be coprime elements. Then the hypersurface ring $$R[t]_{(t)}/(t^2-f_1f_2)$$ is not factorial.
\end{lemma}

\begin{proof}
Let $\mathfrak{n}$ be the maximal ideal of the ring 
$\sO=R[t]_{(t)}/(t^2-f_1f_2)$. Note
that $t$ is irreducible since $t\in \mathfrak{n}\setminus\mathfrak{n}^2$.
To prove the lemma, suppose to the contrary that $\sO$ is factorial.
Since $t^2=f_1f_2$ and $f_1$ and $f_2$ are coprime, we conclude that $f_1$ or $f_2$ is a unit. 
This yields a contradiction.
\end{proof}

\begin{lemma}\label{lemma:factoriality_hypersurface_ring_2}
Let $a,b,c,d,e$, and $f$ be complex numbers. 
The hypersurface ring 
$$\mathbb{C}[x,y,t]_{(x,y,t)}/\big(t^2-ax^2-bxy-cy^2-dx^2y-exy^2-fx^2y^2\big)$$ is not factorial. 
\end{lemma}

\begin{proof}Let $a_1\in \mathbb{C}$ such that
$a=a_1^2$.
The hypersurface $X\subset \mathbb{A}^3$ given by equation 
$$t^2-ax^2-bxy-cy^2-dx^2y-exy^2-fx^2y^2=0$$ contains $(0,0,0)$ in its singular
locus.  Moreover, the line given by equations $$y=t-a_1x=0$$
is a smooth hypersurface on $X$ passing through $(0,0,0)$. Therefore $X$ is not locally factorial at $(0,0,0)$.
\end{proof}

The next result is an immediate consequence of Lemma \ref{lemma:factoriality_hypersurface_ring_2}.

\begin{cor}\label{cor:double_cover_factorial}
Let $X$ be a double cover of $\p^1\times\p^1$ branched along a reduced divisor of type $(2,2)$. Then $X$ is locally factorial if and only if it is smooth.
\end{cor}

\begin{lemma}\label{lemma:lci_factorial}
Let $f \colon X \to Y$ be a finite surjective morphism of degree $2$ with $X$ Cohen-Macaulay and $Y$ smooth.
If $\dim (\Sing(X)) \le \dim X-4$, then $X$ is locally factorial.
\end{lemma}

\begin{proof}
The lemma follows from \ref{say:double_cover} and \cite[Expos\'e XI Corollaire 3.14]{sga2}.
\end{proof}

The following example shows that the statement of Lemma \ref{lemma:lci_factorial} becomes wrong if one relaxes the
assumption on $\dim (\Sing(X))$.

\begin{exmp}
Let $B \subset \p^1\times\p^1\times\p^1$ be defined by equation
$$x_0^2y_0^2z_1^2+z_0^2x_1^2y_1^2+z_0z_1(x_0x_1y_1^2+y_0y_1x_1^2)=0$$ 
and let $X$ be the double cover 
of $\p^1\times\p^1\times\p^1$ branched along $B$. 
Denote by $f\colon X \to \p^1\times\p^1\times\p^1$ the natural morphism.
Set $P=(0,1)\times (0,1)\times (0,1)\in \p^1\times\p^1\times\p^1$ and $Q=f^{-1}(P)\in X$.
A straightforward computation shows that $X$ is normal and singular at $Q$.
The surface $S=f^{-1}\big(\{z_0=0\}\big)$ is the double cover of $\p^1\times\p^1$ branched along 
the divisor defined by equation $x_0^2y_0^2=0$. Hence, $S$ is the union of two copies of $\{z_0=0\}\cong \p^1\times\p^1$ passing through $Q$. 
Therefore $X$ is not locally factorial at $Q$.
\end{exmp}

\begin{lemma}\label{lemma:pi_1_revetement_double}
Let $X$ be a double cover of $\p^1\times\cdots\times\p^1$ branched along a reduced divisor $B$ of type 
$(2,\ldots,2)$. If $\dim \Sing(X)\le \dim X -3$, then 
$\pi_1\big(X\setminus \textup{Sing} (X)\big)\cong \{1\}$.
\end{lemma}

\begin{proof}
Set $n=\dim X$. By Lemma \ref{lemma:degree_double_cover}, $X$ is $\mathbb{Q}$-Fano.
If $n =1$, then $X\cong \p^1$, and hence $\pi_1(X)=\{1\}$ as claimed. Suppose that $n \ge 2$.
Then apply Lemma \ref{lemma:pi_1} below to any projection 
$f_1\colon X \to (\p^1)^{n-1}$ such that $f_1\big(\textup{Supp}(B)\big)=(\p^1)^{n-1}$.
\end{proof}

\begin{lemma}\label{lemma:pi_1}
Let $Y$ be a $\mathbb{Q}$-Fano variety of dimension  $n \ge 2$, and 
let $\pi\colon Y \to T$ be a surjective equidimensional morphism with 
connected fibers onto a normal variety such that 
$\dim \pi\big(\Sing(Y)\big) \le \dim T -2$. 
Then  
$\pi_1\big(T\setminus \textup{Sing} (T)\big)\cong \pi_1\big(Y\setminus \textup{Sing} (Y)\big)$.
\end{lemma} 

\begin{proof}Let $F$ be a general fiber of $\pi$. Then $F$ is a Fano manifold, and therefore $F$ is rationally 
chain connected (see \cite{campana92} and \cite{kmm3}). Consider a smooth general complete intersection curve 
$B \subset T$, and set $Z=\pi^{-1}(B)$. Then $Z$ is smooth, and the morphism $\pi_Z\colon Z \to B$ 
induced by the restriction
of $\pi$ to $Z$ has  rationally chain connected general fibers.
By \cite{ghs03}, we conclude that the scheme theoretic fiber $\pi_{Z}^{-1}(b)$ has a smooth point for each $b\in B$. Therefore, there exists a codimension $\ge 2$ closed subset $G$ of $T$ such that 
\begin{itemize}
\item $\textup{Sing}(T)\cup\pi\big(\textup{Sing}(Y)\big) \subset G$,
\item for each $t\in T \setminus G$, the scheme theoretic fiber $\pi^{-1}(p)$ has a smooth point.
\end{itemize}
Thus, by \cite[Lemma 1.5]{nori}, there is an exact sequence
$$ \pi_1(F) \to \pi_{1}\big(Y \setminus \pi^{-1}(G)\big) \to \pi_1(T \setminus G) \to 1.$$
Now, we have $\pi_1(T \setminus G)\cong\pi_1\big(T\setminus \textup{Sing} (T)\big)$ 
and $\pi_{1}\big(Y \setminus \pi^{-1}(G)\big)\cong \pi_1\big(Y\setminus \textup{Sing} (Y)\big)$
since $S$ and 
$\pi^{-1}(G)$ have codimension $\ge 2$ respectively. 
Fano manifolds being simply connected, we conclude that 
$\pi_1\big(T\setminus \textup{Sing} (T)\big)\cong \pi_1\big(Y\setminus \textup{Sing} (Y)\big)$, proving the lemma.
\end{proof}

\begin{rem}
Example \ref{exmp:del_pezzo_degree4_picard_2} shows that the statements of Lemmata \ref{lemma:pi_1_revetement_double} and \ref{lemma:pi_1} becomes wrong if one relaxes the assumption on $\dim \pi\big(\Sing(Y)\big)$.
\end{rem}

\section{Mori point of view}\label{pullback}

Let $\phi\colon X \to Y$ be an arbitrary surjective morphism of normal projective varieties. The
natural morphism $\phi^*\colon \N^1(Y) \to \N^1(X)$ is injective,
$$\phi^*\big(\Nef(Y)\big)=\Nef(X)\cap \phi^*\big(\N^1(Y)\big),\quad\text{and}\quad \phi^*\big(\Pef(Y)\big)=\Pef(X)\cap \phi^*\big(\N^1(Y)\big).$$

Let now $X$ and $Y$ be normal projective varieties. If $\Nef(X\times Y)=\Pef(X\times Y)$,
then it is easy to see that $\Nef(X)=\Pef(X)$ and $\Nef(Y)=\Pef(Y)$.
We have the following partial converse to the above statement.

\begin{lemma}
Let $X$ and $Y$ be normal projective varieties. 
Suppose that $h^1(X,\sO_X)=0$. If $\Nef(X)=\Pef(X)$ and $\Nef(Y)=\Pef(Y)$, then  
$\Nef(X\times Y)=\Pef(X\times Y)$.
\end{lemma}

\begin{proof}
Since $h^1(X,\sO_X)=0$, we have $\Pic(X\times Y)\cong \Pic(X)\times\Pic(Y)$. It is easy to see that 
$\Nef(X\times Y)\cong \Nef(X)\times\Nef(Y)$, and
$\Pef(X\times Y)\cong \Pef(X)\times\Pef(Y)$. The lemma follows.
\end{proof}

We will make use of the following elementary lemma.

\begin{lemma}\label{lemma:image_nef_egal_psef}
Let $f\colon X \to Y$ be an arbitrary surjective morphism of normal projective varieties. If
$\Nef(X)=\Pef(X)$, then $\Nef(Y)=\Pef(Y)$.
\end{lemma}

\begin{proof}
The lemma follows from the projection formula.
\end{proof}

\begin{rem}
Let $\phi\colon X \to \p^1\times\p^1\times\p^1$ be a double cover branched along a smooth divisor of type 
$(2,2,2)$, and let $\pi\colon X \to \p^1$ be a projection. Then a general fiber $F$ of $\pi$ is a smooth Del Pezzo surface of degree $4$ with $\Nef(F)\subsetneq\Pef(F)$.
\end{rem}

The following observation will prove to be crucial.

\begin{lemma}\label{lemma:nef=psef_contractions}
Let $(X,B)$ be a $\mathbb{Q}$-Fano pair 
with $\mathbb{Q}$-factorial log canonical singularities. Then the following conditions are 
equivalent.
\begin{enumerate}
\item $\Nef(X)=\Pef(X)$.
\item Any effective Cartier divisor on $X$ is semiample. 
\item Any elementary contraction $X \to Y$ satisfies $\dim Y < \dim X$. 
\item Any non-trivial contraction $X \to Y$ satisfies $\dim Y < \dim X$.
\end{enumerate}
\end{lemma}

\begin{proof}
(1) $\Rightarrow$ (2) Let $E$ be an effective Cartier divisor on $X$.
Then $E$ is nef and thus $E-(K_X+B)$ is ample. By \cite[Theorem 13.1]{fujino_mmp}, we conclude that $E$ is semiample.

\medskip

(2) $\Rightarrow$ (3) Let $\phi\colon X \to Y$ be an elementary contraction. We argue by contradiction, and assume that
$\dim Y=\dim X$. Set $C=\phi_*B$. Let $A$ be an ample Cartier divisor on $Y$ such that $K_Y+C+A$ is $\mathbb{Q}$-ample, and
let $m$ be a positive integer such that $mB\in Z^1(X)_\mathbb{Z}$ and 
$h^0\big(Y,\sO_Y(m(K_Y+C+A))\big) \ge 1$.

Suppose first that $\dim \textup{Exc}(\phi)=\dim X-1$. 
Then $F=\textup{Exc}(\phi)$ is 
irreducible, $Y$ is $\mathbb{Q}$-factorial, and $K_X+B=\phi^*(K_Y+C)+aF$ for some rational number 
$a \ge0$. Assume in addition that $ma \in \mathbb{Z}$. Then
$$
\begin{array}{ccll}
h^0\big(X,\sO_X(m(K_X+B+\phi^*A))\big) & =  & 
h^0\big(X,\phi^*\sO_Y(m(K_Y+C+A)+maF)\big)
& \\
& = & h^0\big(Y,\sO_Y(m(K_Y+C+A))\big)  
& \text{ since $ma \in \mathbb{N}$}\\
& \ge  & 1. &\\
\end{array}
$$
Thus, there exists 
an effective Cartier divisor $E \sim_\mathbb{Q}K_X+B+\phi^*A$. 
This yields a contradiction since $K_X+B+\phi^*A$ is not nef.

Suppose now that $\dim \textup{Exc}(\phi) \le \dim X-2$. 
Set $X^\circ=X\setminus \textup{Exc}(\phi)$ and $Y^\circ=Y\setminus \phi(\textup{Exc}(\phi))$. 
Then
$$
\begin{array}{cccl}
h^0\big(X,\sO_X(m(K_X+B+\phi^*A))\big) & =  & h^0\big(X^\circ,\sO_X(m(K_X+B+\phi^*A))\big) 
& \text{ since $\codim\,\textup{Exc}(\phi)\ge 2$}\\
& = & h^0\big(Y^\circ,\sO_Y(m(K_Y+C+A))\big)& \text{ since 
$X^\circ\cong Y^\circ$}\\
& = & h^0\big(Y,\sO_Y(m(K_Y+C+A))\big) & \text{ since $\codim\,\phi(\textup{Exc}(\phi))\ge 2$}\\
& \ge  & 1. &\\
\end{array}
$$
We conclude as before that there exists 
an effective divisor $E \sim_\mathbb{Q}K_X+B+\phi^*A$, yielding a contradiction since 
$K_X+B+\phi^*A$ is not nef. This proves that $\dim Y < \dim X$.

\medskip

(3) $\Rightarrow$ (4) is obvious.

\medskip

(4) $\Rightarrow$ (1) Let $E$ be an effective Cartier divisor. Suppose that $E$ is not nef. Then there exists an extremal ray 
$R \subset \NE(X)$ such that $E \cdot C<0$ for every curve $C$ with $[C]\in R$. Let $\phi\colon X \to Y$ be the corresponding contraction. Then we must have $\textup{Exc}(\phi)\subset \textup{Supp}(E)$, yielding a contradiction, and completing the proof of the lemma.
\end{proof}

To prove Theorem \ref{thm:main}, we will argue by induction on $\dim X$. We will make use of the following result.

\begin{lemma}\label{lemma:target_log_Fano}
Let $(X,B)$ be a $\mathbb{Q}$-Fano pair 
with log terminal (respectively, log canonical) singularities. Let $\phi\colon X \to Y$ be any contraction. Then 
there exists an effective $\mathbb{Q}$-divisor $B_Y$ on $Y$ such that 
$(Y,B_Y)$ is $\mathbb{Q}$-Fano with log terminal (respectively, log canonical) singularities.
\end{lemma}

\begin{proof}Let $A$ be an ample $\mathbb{Q}$-divisor on $Y$, $0<\varepsilon \ll 1$ a rational number, and 
$C \sim_\mathbb{Q} -(K_X+B)-\varepsilon \phi^*A$ an effective $\mathbb{Q}$-divisor such that 
$(X,B+C)$ is klt. 
Suppose that $(X,B)$ has log terminal singularities.
By \cite[Theorem 4.1]{ambro_lc_trivial} applied to $(X,B+C)$, there exists an effective $\mathbb{Q}$-divisor $B_Y$ on $Y$ such that
$(Y,B_Y)$ is klt and
$K_Y+B_Y\sim_\mathbb{Q} -\varepsilon A$. If 
$(X,B)$ has log canonical singularities, then one only needs to replace the use of \cite[Theorem 4.1]{ambro_lc_trivial}
with \cite[Therorem 3.4]{fujino_gongyo_adjunction}.
\end{proof}

The same argument used in the proof of \cite[Lemma 5.1.5]{kmm} shows that the following lemma holds.
One only needs to replace the use of \cite[Lemma 3.2.5]{kmm} with \cite[Theorem 1.1]{fujino_mmp}.

\begin{lemma}\label{lemma:factoriality_contraction}
Let $(X,B)$ be a pair with log canonical singularities, and let $\varphi\colon X \to Y$ be an elementary Mori contraction
with $\dim Y < \dim X$. 
If $X$ is locally factorial (respectively, $\mathbb{Q}$-factorial), then $Y$ is locally factorial (respectively, $\mathbb{Q}$-factorial).
\end{lemma}

\begin{rem}
Let $(X,B)$ be a pair with log canonical singularities. Then the conclusion of
Lemma \ref{lemma:factoriality_contraction} holds for any quasi-elementary Mori contraction
$\phi\colon X \to Y$ with $\dim Y<\dim X$. We refer to \cite{cinzia_quasi_elementary} for the definition of quasi-elementary Mori contractions.
\end{rem}

\begin{cor}\label{corollary:factoriality_contraction}
Let $(X,B)$ be a $\mathbb{Q}$-Fano pair 
with log canonical singularities satisfying $\Nef(X)=\Pef(X)$. 
Let $V$ be a face of $\NE(X)$, and denote by $\phi\colon X \to Y$ the corresponding contraction. If $X$ is locally factorial, then so is $Y$.
\end{cor}

\begin{proof}We argue by induction on $\rho(X/Y)=\rho(X)-\rho(Y)\ge 0$. If $\rho(X/Y)=0$, then $X\cong Y$ and there is nothing to prove.
Suppose that $\rho(X/Y) \ge 1$.
Let $R \subset V$ be an extremal ray of $\NE(X)$, and denote by $\psi\colon X \to Z$ the corresponding contraction. 
There exists a morphism
$\xi\colon Z \to Y$ such that $\varphi = \xi\circ\psi $.
From Lemma \ref{lemma:image_nef_egal_psef}, we conclude that
$\Nef(Z)=\Pef(Z)$.
By Lemma \ref{lemma:target_log_Fano}, there exists an effective $\mathbb{Q}$-divisor $B_Z$ on $Z$ such that 
$(Z,B_Z)$ is $\mathbb{Q}$-Fano with log canonical singularities.
By Lemma \ref{lemma:nef=psef_contractions}, we have $\dim Z<\dim Y$, and Lemma \ref{lemma:factoriality_contraction} tells us
that $Z$ is locally factorial. Notice that
$\rho(Z/Y)=\rho(X/Y)-1<\rho(X/Y)$. The lemma follows.
\end{proof}

It follows from \cite[Theorem 2.2]{wisn_fano} that a Fano manifold with 
$\Nef(X)=\Pef(X)$ satisfies $\rho(X) \le \dim X$. We now extend this result to mildly singular varieties.

\begin{lemma}\label{lemma:nef_pseff_picard}
Let $(X,B)$ be a $\mathbb{Q}$-Fano pair 
with log canonical singularities satysfying $\Nef(X)=\Pef(X)$. 
\begin{enumerate}
\item Then $\rho(X) \le \dim X$.
\item Suppose moreover that $\rho(X) = \dim X$. Then $\NE(X)$ is simplicial, and the following holds. 
Let $V$ be a face of $\NE(X)$, and denote by $\phi\colon X \to Y$ the corresponding contraction. Then 
$\phi$ is equidimensional and
$\dim Y=\dim X - \dim V = \rho(Y)$.
\end{enumerate}
\end{lemma}

\begin{proof}
Set $n=\dim X$.
Suppose that $\rho(X) = \dim \NE(X) \ge n$. 
Let $R_1,\ldots,R_{n-1}$ be extremal rays of $\NE(X)$ such that
$V_i=R_1+\cdots +R_i$
is a face of $\NE(X)$ with $\dim (R_1+R_2+\cdots+R_i)=i$ for each $1 \le i \le n-1$.
Denote by $\phi_i\colon X \to Y_i$ the 
contraction of $V_i$. 
We have $\dim Y_i \le n-i$ by Lemma \ref{lemma:nef=psef_contractions}(4).
Let $i_0$ be the smallest number $1 \le i \le n-1$ such that $\dim Y_{i_0}\le 1$. 
By \cite[Theorem 16.4(3)]{fujino_mmp}, we have 
$\rho(Y_i)=\rho(X)-i$, and therefore
$$\rho(X)-n+1 \le \rho(X)-i_0 = \rho(Y_{i_0}) \le 1.$$
We conclude that $\rho(X)=n$ and $i_0=n-1$, proving (1).

We proceed to prove (2). Suppose that $\rho(X)=\dim X$. 

We first show that $\NE(X)$ is simplicial.
Suppose otherwise and consider a $2$-dimensional face $W$ of $\NE(X)$ such that $V_{n-1}\cap W =\{0\}$. Denote by $\psi\colon X \to T$ the corresponding contraction. Let 
$F$ and $G$ be fibers of $\phi_{n-1}$ and $\psi$ respectively. Since $\dim Y_{n-1}=1$ and $\dim T \le n-2$, we must have $\dim(F \cap G) \ge 1$, yielding a contradiction. This proves that 
$\NE(X)$ is simplicial.

Let $V$ be a face of $\NE(X)$, and denote by $\phi\colon X \to Y$ the corresponding contraction.
We have $\rho(Y)= \rho (X) - \dim V = \dim X - \dim V$. 
Thus there exists a face
$W$ of $\NE(X)$ with $\dim W = \rho(Y)$ and $V \cap W= \{0\}$.
Denote by $\psi\colon X \to T$ the corresponding contraction. 
By Lemma \ref{lemma:nef=psef_contractions}(4), we must have 
$$\dim Y \le \dim X - \dim V = \rho(Y)$$
and 
$$\dim T \le \dim X - \dim W = \dim X - \rho (Y).$$ 
Since $V \cap W= \{0\}$, 
we conclude that $\phi$ and $\psi$ are equidimensional,
$\dim Y = \dim X - \dim V $,
and $\dim T = \dim X - \dim W$.
This completes the proof of the lemma.
\end{proof}

\begin{rem}\label{remark:conic_bundle}
In the setup of the proof of Lemma \ref{lemma:nef_pseff_picard}, 
suppose moreover that $X$ is smooth. 
Then $Y_i$ is smooth and $\phi_i$ is a conic bundle
for each $1 \le i \le n$ by \cite[Theorem 3.1]{ando}.  
\end{rem}

\begin{rem}One might ask whether Lemma \ref{lemma:nef_pseff_picard} holds for a larger class of singular varieties. What we actually proved is the following.
Let $X$ be a normal projective variety with 
rational polyhedral Mori cone $\NE(X)$. Suppose that
every face $V$ of $\NE(X)$ corresponds to a surjective morphism with connected fibers 
$\phi\colon X \to Y$ onto a normal projective variety such that $\phi$ contracts precisely those curves on $X$ with class in $V$. Suppose furthermore that $\rho(Y)=\rho(X)-\dim V$ and $\dim Y < \dim X$ for each face $V$ of $\NE(X)$. If $\Nef(X)=\Pef(X)$, then 
$X$ satisfies the conclusions of Lemma \ref{lemma:nef_pseff_picard}.
\end{rem}

The same argument used in the proof of Lemma \ref{lemma:nef_pseff_picard} above shows that the following lemmata hold.

\begin{lemma}\label{lemma:contraction}
With the assumptions as in Lemma \ref{lemma:nef_pseff_picard}, suppose furthermore that there is a contraction $\phi\colon X \to Y$ such that $\dim Y=\rho (Y)$.
Then, there exists a contraction
$\psi\colon X \to T$ such that the induced morphism 
$(\phi,\psi)\colon X \to Y \times T$ is finite and surjective.
In particular, $\phi$ and $\psi$ are equidimensional.
\end{lemma}

\begin{lemma}\label{lemma:nef_pseff_maximal_contraction}
With the assumptions as in Lemma \ref{lemma:nef_pseff_picard}, suppose furthermore that
there is an elementary contraction 
$\phi\colon X \to Y$ with $\dim Y = \rho (X)-1=\rho(Y)$. 
Denote by $R$ the corresponding extremal ray.
Then $\NE(X)$ is simplicial, and the following holds. 
Let $W$ be a face of $\NE(X)$, and 
denote by $\psi\colon X \to T$ the corresponding contraction. Then $\psi$ is equidimensional, and either
$\dim T=\dim X - \dim W$  if $R \not\subset W$ or $\dim T=\rho(X)-\dim W = \rho(T)$  if $R \subset W$.
\end{lemma}

\section{Fibrations on $\mathbb{Q}$-Fano threefolds with canonical Gorenstein singularities}

In this section we provide a technical tool for the proof of the main results.

\begin{say}\label{reminder:minimal_degree} 
Let $N \ge 2$ be an integer. We say that $Y \subset \p^N$ is a \textit{variety of minimal degree} if $Y$ is nondegenerate and $\deg(Y)=\codim\,Y +1$.
Surfaces of minimal degree were classified by Del Pezzo. Bertini then obtained a similar classification for varieties of any dimension. 
A \textit{rational normal scroll is} is a 
If $Y \subset \p^N$ is a variety of minimal degree and $\codim\, Y\ge 2$, 
then $Y$ is either a rational normal scroll or a cone over the Veronese surface in $\p^2\subset \p^5$.
A \textit{rational normal scroll} is a cone over a smooth linearly normal variety fibered over $\p^1$ by linear spaces. Note that the Veronese surface contains no lines, and thus a cone over the Veronese surface cannot contain a linear space of codimension $1$. A rational normal scroll contains a pencil of linear spaces of codimension $1$. This pencil is unique if and only if $Y$ is not a cone over $\p^1\times\p^1\subset \p^3$.

\end{say}

\begin{say}\label{reminder:gorenstein_3_fano}
Let $X$ be a $3$-dimensional $\mathbb{Q}$-Fano variety with 
Gorenstein canonical singularities. 
Then we have
$$h^0\big(X,\sO_X(-K_X)\big)=-\frac{1}{2}K_X^3+3.$$
Set $N=-\frac{1}{2}K_X^3+2$, and notice that $N\ge 3$. Denote by $\phi\colon X \dashrightarrow \p^N$ the rational map given by the complete linear system
$|-K_X|$, and set $Y=\phi(X)$.

The classification of $3$-dimensional $\mathbb{Q}$-Fano varieties $X$ with 
Gorenstein canonical singularities satisfying $\textup{Bs}(-K_X)\neq \emptyset$ was established in  
\cite{jahnke_radloff}. Those with $\textup{Bs}(-K_X)= \emptyset$ and $-K_X$ not very ample were classified in 
\cite[Theorem 1.5]{cheltsov_trigonal}. We do not include the classification here. Instead, we state the properties that we need.
If $\textup{Bs}(-K_X)\neq \emptyset$, then 
$Y \subset \p^N$ is a variety of minimal degree with
$\dim Y=2$, and the rational map $\phi\colon X \dashrightarrow Y$ has connected fibers.
If $\textup{Bs}(-K_X)= \emptyset$ and $-K_X$ is not very ample, then 
$Y \subset \p^N$ is a variety of minimal degree, and the finite morphism $\phi\colon X \to Y$ has degree 2.
\end{say}

We are now ready to state and prove the main result of this section.

\begin{lemma}\label{lemma:anticanonical_map}
Let $X$ be a $3$-dimensional $\mathbb{Q}$-Fano variety with 
Gorenstein canonical singularities. Then there is at most one fibration 
$f\colon X \to \p^1$
with general fiber $S$ satisfying $K_S^2\le 2$.
\end{lemma}

\begin{proof}
Denote by $\phi\colon X \dashrightarrow \p^N$ the rational map given by the complete linear system $|-K_X|$, and set $Y=\phi(X)\subset \p^N$.

Let $f\colon X \to \p^1$ be a surjective morphism with connected fibers, and general fiber $S$. Suppose that 
$K_S^2 \le 2$. We will show that
there exists a rational fibration $g \colon Y \dashrightarrow \p^1$ by linear spaces 
such that $g \circ \phi$ factors through $f$.
Notice that $S$ has 
canonical singularities. Denote by $\psi \colon S \dashrightarrow \p^M$  
the rational map given by the complete linear system
$|-K_S|$, where $M+1=h^0\big(S,\sO_S(-K_S)\big)=K_S^2+1$. 
Then $\textup{Bs}(-K_S)\neq \emptyset$ if and only if $K_S^2=1$.
If $K_S^2=2$, then $\psi\colon S \to \p^2$ is a double cover.
We conclude that $-K_S$ is not very ample, and therefore neither is $-K_X$.

The restriction map $r \colon H^0\big(X,\sO_X(-K_X)\big) \to 
H^0\big(S,\sO_S(-K_S)\big)$
induces a commutative diagram
$$
\xymatrix{
    X \ar@{-->}[r]^{\phi}  & {\p^N}  \\
    S \ar@{-->}[r]_{\psi} \ar[u] & {\p^M} \ar@{-->}[u]_{r^\vee}
  }
$$
Notice that $\psi$ is dominant since $K_S^2\le 2$ by assumption. Therefore, $\phi(S)\subset Y\subset \p^N$ is a linear subspace of dimension
$\le 2$.

Suppose first that $\textup{Bs}(-K_X)\neq \emptyset$. By \ref{reminder:gorenstein_3_fano}, 
$\dim Y=2$ and 
$Y \subset \p^N$ is a variety of minimal degree. If $K_S^2=2$, then $Y \subset \p^N$ 
is a plane. But this contradicts \ref{reminder:minimal_degree}.
We conclude that $K_S^2=1$, and 
$\phi(S)\subset Y$ is a line. Notice that the $\phi(S)$'s yield a covering family of lines on $Y$. By \ref{reminder:minimal_degree}, 
there is a rational fibration $s \colon Y \dashrightarrow \p^1$ by lines
such that $f=s \circ \phi$. Suppose that there is a fibration
$g\colon X \to \p^1$ with general fiber $T$ satisfying $K_T^2\le 2$ and $g \neq f$. Then, as before, 
there is a rational fibration $t \colon Y \dashrightarrow \p^1$ by lines
such that $t=t \circ \phi$ and $t\neq s$. By \ref{reminder:minimal_degree}, we must have 
$Y \cong \p^1\times \p^1 \subset \p^3$. It follows that $\phi = (f,g)$ but this contradicts the fact
that $\textup{Bs}(-K_X)\neq \emptyset$. 

Suppose then that $\textup{Bs}(-K_X)= \emptyset$. Notice that $\dim Y=3$.
Then we must have
$\textup{Bs}(-K_S)=\emptyset$, $K_S^2=2$, and $h^0\big(S,\sO_S(-K_S)\big)=3$.
We claim that the restriction map 
$r \colon H^0\big(X,\sO_X(-K_X)\big) \to 
H^0\big(S,\sO_S(-K_S)\big)$ is surjective.
Suppose otherwise. Then $|-K_S|$ contains a base point free sub-linear system 
of dimension $1$, yielding a contradiction since $-K_S$ is ample.
We conclude that
$\phi(S)\subset Y\subset \p^N$ is a $2$-dimensional linear subspace and $\phi_{|S}=\psi$. In particular, 
$\deg(\phi)=\deg(\psi)=2$. Therefore, 
there is a rational fibration $s \colon Y \dashrightarrow \p^1$ by $2$-dimensional linear subspaces, 
such that $f=s \circ \phi$.
By \ref{reminder:gorenstein_3_fano}, either
$Y\cong \p^3$, or
$Y \subset \p^N$ is a rational normal scroll. 
If $Y \cong \p^3$, then $s \colon Y \dashrightarrow \p^1$ is a linear projection. Thus 
$f=s \circ \phi$ cannot be a morphism, yielding a contradiction.
If $Y \subset \p^N$ is a rational normal scroll, then $s\colon Y \dashrightarrow \p^1$ is unique so that $f$ is unique as well. This completes the proof of the lemma.
\end{proof}

\begin{exmp}
Let $S$ be a smooth Del Pezzo surface with $K_S^2\le 2$, and set $X=S \times \p^1$. 
Denote by $\phi\colon X \dashrightarrow \p^N$ the rational map given by the complete linear system
$|-K_X|$, and set $Y=\phi(X)$.
Then $Y\cong \p^1\times\p^1$ (respectively, $Y \cong \p^2\times\p^1$) if $K_S^2=1$ (respectively, $K_S^2=2$). 
Any fiber $T \cong S$ of the natural projection $X \to \p^1$ satisfies $K_T^2 \le 2$.
\end{exmp}

\section{Fano varieties of dimension $\le 3$ with $\Nef(X)=\Pef(X)$}

The classification of Fano manifolds of dimension $\le 3$ with $\Nef(X)=\Pef(X)$ will be presented in this section.

\begin{prop}\label{proposition:nef_psef_surfaces}
Let $S$ be a locally factorial Del Pezzo surface satisfying $\Nef(S)=\Pef(S)$.
Then either $S \cong \p^2$, or $S \cong \p^1\times \p^1$.
\end{prop}

\begin{proof}Let $\phi\colon S \to \p^1$ be an 
elementary contraction. The irreducible components of fibers of $\varphi$ are smooth rational curves
by \cite[Corollary 1.9]{andreatta_wisniewski_contractions}. Since $S$ is locally factorial, we conclude that $S$ is smooth. The lemma follows easily from the Enriques-Kodaira classification of smooth complex proper surfaces.
\end{proof}

The following example shows that the statement of Proposition \ref{proposition:nef_psef_surfaces} does not hold for
normal surfaces with canonical singularities.  

\begin{exmp}\label{exmp:del_pezzo_degree4_picard_2}
Let $T$ be the Del Pezzo surface of degree $K_S^2=4$ with $4$ singular points of type $A_1$
given by equations $$x_0x_1+x_2^2=x_3x_4+x_2^2=0$$ in $\p^4$. Then $T$ admits a double cover $\p^1\times\p^1 \to T$ ramified only at the singular points (see \cite[Proposition 0.4.2]{cossec_dolgachev}). In particular, we have $\Nef(T)=\Pef(T)$.
\end{exmp}

Fano manifolds with $\dim X=3$ and $\rho(X) \ge 2$ were classified by Mori and Mukai in \cite{mori_mukai}. 
The classification of those with 
$\Nef(X)=\Pef(X)$ follows easily.

\begin{thm}Let $X$ be a Fano manifold of dimension $3$ with $\rho(X) \ge 2$. If $\Nef(X)=\Pef(X)$, then 
$X$ is isomorphic to one of the following.
\begin{enumerate}
\item $\p^1\times\p^2$.
\item A double cover of $\p^1\times\p^2$ branched along a divisor of type $(2,2)$.
\item A double cover of $\p^1\times\p^2$ branched along a divisor of type $(2,4)$.
\item A double cover of $Y=\p_{\p^2}(T_{\p^2})$ whose branched locus is a member of $|-K_Y|$.
\item A hypersurface of type $(1,1)$ in $\p^2\times\p^2$.
\item A hypersurface of type $(1,2)$ in $\p^2\times\p^2$.
\item A hypersurface of type $(2,2)$ in $\p^2\times\p^2$.
\item A double cover of $\p^1\times\p^1\times\p^1$ branched along a divisor of type $(2,2,2)$.
\end{enumerate}

\end{thm}

\section{Finite morphisms onto $\p^1\times\cdots\times\p^1$}

In this section we prove Proposition \ref{proposition:intro}.

\begin{lemma}\label{lemma:hodge_index_theorem}
Let $S$ be a $\mathbb{Q}$-Fano surface with Gorenstein singularities, and let 
$C_1$ and $C_2$ be two curves contained in the smooth locus of $S$ satisfying 
$C_1^2=C_2^2=0$. Then 
$$2K_S^2(C_1\cdot C_2) \le (K_S\cdot C_1+K_S \cdot C_2)^2.$$
\end{lemma}

\begin{proof}Let $\mu\colon T \to S$ be the minimal resolution of singularities of $S$, and denote by $\tilde{C}_1$ and $\tilde{C}_2$ the strict transforms in $T$ of $C_1$ and $C_2$ respectively.
Using the projection formula, we obtain
$(\mu^*K_S)^2=K_S^2>0$. Therefore,
by the Hodge Index Theorem, we have
$$(\mu^*K_S)^2(\tilde{C}_1+\tilde{C}_2)^2  \le \big(\mu^*K_S\cdot(\tilde{C}_1+\tilde{C}_2)\big)^2.$$
The assumption that $C_1$ and $C_2$ are contained in the smooth locus of $S$
yields $(\mu^*K_S)^2(\tilde{C}_1+\tilde{C}_2)^2 = 2K_S^2(C_1\cdot C_2)$. Finally,  
$\mu^*K_S\cdot(\tilde{C}_1+\tilde{C}_2)= K_S\cdot (C_1+C_2)$
by the projection formula, completing the proof of the lemma.
\end{proof}

\begin{prop}\label{proposition:finite_morphism_quadric_bis}
Let $S$ be a $\mathbb{Q}$-Fano surface with Gorenstein singularities satisfying $K_S^2 \ge 3$, 
and let $f_1\colon S \to \p^1$ and $f_2\colon S \to \p^1$ be two surjective morphisms with connected fibers.
Suppose that the induced morphism $f=(f_1,f_2) \colon S \to \p^1\times\p^1$ is finite.
Then either $K_S^2=4$ and $f$ makes $S$
a double cover of $\p^1\times \p^1$ branched along a reduced divisor of type $(2,2)$, or $K_S^2=8$ and $f$ induces
an isomorphism
$S\cong \p^1\times \p^1$. 
\end{prop}

\begin{proof}
Let $C_1$ and $C_2$ be general fibers of $f_1$ and $f_2$ respectively.
Notice that $-K_S \cdot C_1=-K_S \cdot C_2=2$.
By Lemma \ref{lemma:hodge_index_theorem}, we have
$$6 \, C_1\cdot C_2 \le 2K_S^2(C_1\cdot C_2) \le (K_S\cdot C_1+K_S\cdot C_2)^2=16,$$
and thus
$\deg(f)=C_1 \cdot C_2 \le 2$. If $\deg(f) = 1$, then $S \cong \p^1\times\p^1$ and $K_S^2=8$. If 
$\deg(f) = 2$, then $f$ makes $S$ a double cover of $\p^1\times\p^1$ branched along a reduced divisor of type $(2,2)$ by Lemma \ref{lemma:degree_double_cover}. A straightaforward computation then shows that $K_S^2=4$.
\end{proof}

\begin{prop}\label{proposition:characterization_double_cover}
Let $X$ be a $\mathbb{Q}$-Fano variety of dimension $n \ge 3$ with Gorenstein canonical singularities. 
Let $f\colon X \to (\p^1)^n$ be a finite morphism 
of degree $\deg(f) >1$
such that each projection 
$X \to (\p^1)^{n-1}$ has connected fibers. Then $f$ makes $X$ a double cover of 
$(\p^1)^n$ branched along a reduced divisor of type $(2,\ldots,2)$.   
\end{prop} 

\begin{proof}By Lemma \ref{lemma:degree_double_cover}, it suffices to show that $\deg(f)=2$.

Suppose first that $n=3$, and let $f_1\colon X \to \p^1$ and $f^1\colon X \to \p^1\times \p^1$ such that $f=(f_1,f^1)$. Let $S$ be a general fiber of $f_1$. Notice that $S$ is a Del Pezzo surface. Moreover, the restriction of $f^1$ to $S$ induces a finite morphism 
${f^1}_{|S}\colon S \to \p^1\times \p^1$ such that each projection $S \to \p^1$ has connected fibers, and $\deg(f)=\deg({f^1}_{|S})$. By Lemma \ref{lemma:anticanonical_map}, we may assume without loss of generality that $K_S^2\ge 3$. We conclude that $\deg(f)=\deg({f^1}_{|S})= 2$
by Proposition \ref{proposition:finite_morphism_quadric_bis}. This proves the proposition in the case when $n=3$. 

Suppose from now on that $n \ge 4$ is arbitrary, and let $f_1\colon X \to (\p^1)^{n-3}$ and $f^1\colon X \to (\p^1)^3$ such that $f=(f_1,f^1)$. Let $F$ be a general fiber of $f_1$. Then $F$ is 
a $\mathbb{Q}$-Fano variety of dimension $3$ with Gorenstein canonical singularities.  
The restriction of $f^1$ to $F$ induces a finite morphism $F \to (\p^1)^3$
such that each projection $F \to (\p^1)^2$ has connected fibers, and $\deg(f)=\deg(f^1)$. We conclude that 
$\deg(f)=\deg(f^1)=2$ from the previous case, completing the proof of the proposition.
\end{proof}

\begin{rem}Proposition \ref{proposition:finite_morphism_quadric} shows that the statement of
Proposition \ref{proposition:characterization_double_cover} above becomes wrong when $\dim X=2$. 
\end{rem}

\begin{proof}[Proof of Proposition \ref{proposition:intro}]
Let $X$ be a $\mathbb{Q}$-Fano variety of dimension $n \ge 3$ with Gorenstein canonical singularities. 
Suppose that $\Nef(X)=\Pef(X)$ and $\rho(X)=n$. Suppose furthermore
that each non-trivial contraction $X \to Y$ with $\dim Y < n$ satisfies $Y\cong \p^1\times \cdots \times \p^1$.

By Lemma \ref{lemma:nef_pseff_picard}, the Mori cone $\NE(X)$ is simplicial.
Let $V_1,\ldots,V_n$ be the codimension $1$ faces of $\NE(X)$, and denote by $f_i\colon X \to \p^1$ 
the contraction of $V_i$. Set $f=(f_1,\ldots,f_n) \colon X \to (\p^1)^n$. Notice that $f$ is a finite morphism. If $\deg(f)=1$, then
$X \cong (\p^1)^n$. Suppose from now on that $\deg(f)>1$. Let $g\colon X \to (\p^1)^{n-1}$ be a projection
with Stein factorization $\phi \colon X \to Y\cong (\p^1)^{n-1}$. Then
$\phi$ is given by $n-1$ morphisms with connected fibers $X \to \p^1$, and hence $g=\phi$. 
Therefore, Proposition \ref{proposition:characterization_double_cover} applies. 
Using \ref{say:double_cover},
we conclude that $f$ makes $X$ a double cover of $(\p^1)^n$ branched along a reduced divisor of type $(2,\ldots,2)$.
This completes the proof of the theorem.
\end{proof}

\section{Finite morphisms onto $\mathbb{Q}$-Fano cyclic covers of $\p^1\times\cdots\times\p^1$}

In this section we prove Theorem \ref{thm_intro:finite_onto_double_cover}. 

\begin{lemma}\label{lemma:image_fano_finite_morphism}
Let $f \colon Z \to Y$ be a finite surjective morphism of normal projective $\mathbb{Q}$-Gorenstein
varieties. If $-K_Z$ is big, then so is $-K_Y$.
\end{lemma}

\begin{proof}
Denote by $R$ the ramification divisor of $f$. The Hurwitz formula
$K_Z \sim f^*K_Y+R$ shows that $-f^*K_Y \sim -K_Z+R$ is big.
Let $A$ be an ample Cartier divisor on $Y$, $E$ an effective divisor on $X$, 
and $m$ a positive integer such that 
$m(-f^*K_Y)\sim_\mathbb{Q}f^*A + E$. Then
$-K_Y\sim_\mathbb{Q}\frac{1}{m}A+\frac{1}{m\deg(f)}f_*E$
by the projection formula. Hence $-K_Y$ is big, proving the lemma.
\end{proof}

\begin{lemma}\label{lemma:double_cover_connected}
Let $Z$ be a $\mathbb{Q}$-Fano variety of dimension $n\ge 2$, $Y$ a double cover of $(\p^1)^n$ branched along a reduced divisor $B=\sum_{j\in J} B_j$ of type
$(2,\ldots,2)$, and $f\colon Z \to Y$ a finite morphism. Consider a projection $\pi_1\colon Y \to \p^1$, and set $f_1=\pi_1\circ f\colon Z\to \p^1$. 
Suppose that $B_j$ is not contracted by $\pi_1$ for each $j\in J$.
Then $f_1$ has connected fibers.
\end{lemma}

\begin{proof}Denote by $\pi\colon Y\to (\p^1)^n$ the natural morphism, and let 
$\pi^1\colon Y \to (\p^1)^{n-1}$ such that $\pi=(\pi_1,\pi^1)$.
The Stein factorization $g_1\colon Z \to \p^1$ of $f_1=\pi_1\circ f$ fits into a commutative diagram
$$
\xymatrix{
    Z \ar@/^1pc/[rr]^{f} \ar[r] \ar[d]^{g_1} & Y_1 \ar[r]_{\beta_1}\ar[d]^{\alpha_1} & Y \ar@/^1pc/[rr]^{\pi^1}\ar[d]^{\pi_1}\ar[r]_{\pi} & (\p^1)^n\ar[r]\ar[d] & (\p^1)^{n-1}\\
    \p^1 \ar@{=}[r] & \p^1\ar[r]_{h_1} &  \p^1 \ar@{=}[r] & \p^1. &
  }
$$
where $(\alpha_1,\pi_1\circ\beta_1)\colon Y_1 \to (\p^1)^n$ is the double cover of 
$(\p^1)^n$ branched along $(h_1,\textup{Id}_{(\p^1)^{n-1}})^*B$.
The variety $Y_1$ is Cohen Macaulay by \ref{say:double_cover}. 
All fibers of $\pi_1$ are generically reduced since $\pi_1$ does not contract any 
irreducible component of $B$. This implies that all fibers of $\alpha_1$ are generically reduced as well, and hence $\dim\big(\Sing(Y_1)\big)\le n-2$. We conclude that $Y_1$ is normal.
A straightforward computation gives
$$\sO_{Y_1}(-K_{Y_1})\cong 
\alpha_1^*\sO_{\p^1}\big(2-\deg(h_1)\big)\otimes 
(\pi^1\circ\beta_1)^*\big(\sO_{\p^1}(1)\boxtimes\cdots\boxtimes\sO_{\p^1}(1)\big).$$
By Lemma \ref{lemma:image_fano_finite_morphism},
$-K_{Y_1}$ is big. Hence, we must have $\deg(h_1)=1$, proving the lemma.
\end{proof}

\begin{prop}\label{proposition:double_cover_ter}
Let $S$ be a $\mathbb{Q}$-Fano surface with Gorenstein singularities satisfying $K_S^2\ge 3$,
$T$ a double cover of $\p^1\times\p^1$ branched along a reduced divisor $B=\sum_{j\in J} B_j$ of type
$(2,2)$, and $f\colon S \to T$ a finite morphism. Suppose that $B_j$ is ample 
for each $j\in J$. Then $\deg(f)=1$.
\end{prop}

\begin{proof}
Denote by $\pi\colon T \to \p^1\times\p^1$ the natural morphism, and denote by 
$\pi_1\colon T \to\p^1$ and
$\pi_2\colon T \to\p^1$ the projections. Set $g_i=\pi_i\circ f \colon S \to \p^1$, and
$g=(g_1,g_2)\colon S \to \p^1\times\p^1$ so that $g=\pi\circ f$.
By Lemma \ref{lemma:double_cover_connected} above, 
$g_i$ has connected fibers.  By Lemma \ref{lemma:hodge_index_theorem}, we have $\deg(g)K_S^2 \le 8$.
Since $\deg(g)=2\deg(f)$, we obtain $$\deg(f) \le \frac{4}{K_S^2} \le \frac{4}{3}.$$
Thus $\deg(f)=1$, proving the proposition.
\end{proof}

The following example show that the statement of Proposition \ref{proposition:double_cover_ter} becomes wrong if one relaxes the assumption on $B$.

\begin{exmp}
Let $C_1$ and $C_2$ be smooth rational curves, and let $c\in C_1$. 
Denote by $p_1$ the projection $C_1\times C_2 \to C_1$, and 
set $B_1=p_1^{-1}(c)$. Let $B_2$ be a general member in the linear system 
$|\sO_{C_1}(1)\boxtimes\sO_{C_2}(2)|$. 
Set $B=B_1+B_2$, and let $\pi\colon S \to C_1\times C_2$ be the $2$-fold cyclic cover branched along 
$B$. Then $S$ is a normal surface of degree $K_S^2=4$ with $2$ singular points of type $A_1$.
Let $c'\in C_1$ with $c'\neq c$, and let
$h\colon \bar C_{1} \to C_1$ be the $2$-fold cyclic cover branched along
$c+c'$. Let $\bar S$ be the normalization of the fiber product 
$S \times_{C_1} \bar C_1$ with natural morphisms $\alpha\colon \bar{S}\to \bar{C}_1$ and 
$\beta\colon \bar{S} \to S$. A straightforward computation shows that $\beta$ is \'etale at each point
in $\bar{S}\setminus \big((\pi\circ\beta)^{-1}(B_1\cap B_2) \cup (h\circ\alpha)^{-1}(c') \big)$.
Moreover, $\beta$ ramifies along 
$(h\circ\alpha)^{-1}(c')$ with multiplicity $2$. Therefore, 
$K_{\bar{S}}\sim \beta^*K_S+(h\circ\alpha)^{-1}(c')$. We conclude that $\bar{S}$ is $\mathbb{Q}$-Fano with Gorenstein singularities.
Notice that 
$\pi_1\big(S\setminus \Sing(S)\big)\cong \pi_1\big(\p^1\setminus \{c\}\big)=\{1\}$
by Lemma \ref{lemma:pi_1}. Moreover, by 
Proposition \ref{proposition:finite_morphism_quadric_bis}, we have $\bar{S}\cong \p^1\times\p^1$.
\end{exmp}

In the setup of Proposition \ref{proposition:double_cover_ter}, suppose moreover that
$S$ and $T$ are smooth. Then we show that the conclusion still holds when $K_S^2\le 2$.

\begin{prop}\label{proposition:finite_morphism_degree4}
Let $S$ and $T$ be smooth Del Pezzo surfaces, and let $f\colon S \to T$ be a finite morphism.
Suppose that $K_T^2=4$. Then $\deg(f)=1$.
\end{prop}

\begin{proof}
By \cite[Theorem 8.6.2]{dolgachev_book}, $T$ is isomorphic to the complete intersection given by equations
$$x_0^2+x_1^2+x_2^2+x_3^2+x_4^2=a_0x_0^2+a_1x_1^2+a_2x_2^2+a_3x_3^2+a_4x_4^2=0$$ 
in $\p^4$, where
$a_i\neq a_j$ for $i\neq j$. We may assume without loss of generality that 
$a_4\neq 0$. Thus, the projection from $(0,0,0,0,1)$ induces a finite morphism 
$\pi\colon T \to Q$
of degree $2$
onto the smooth quadric surface $Q\cong\p^1\times\p^1$ given by equation 
$$(a_0-a_4)x_0^2+(a_1-a_4)x_1^2+(a_2-a_4)x_2^2+(a_3-a_4)x_3^2=0.$$
Moreover, $\pi$ is branched along a smooth curve of type $(2,2)$. 
Denote by 
$\pi_i\colon Y \to\p^1$ with $i\in\{1,2\}$ the natural projections, and set
$g_i=\pi_i\circ f \colon S \to \p^1$. 
By Lemma \ref{lemma:double_cover_connected},  
$g_i$ has connected fibers. Let $F$ be a component of a reducible fiber of $\pi$, and set
$E=f^{-1}(F)$. Then $E$ and $F$ are exceptional curves on $S$ and $T$ respectively, and 
$f^*F=E$. Taking squares gives $\deg(f)=1$, proving the proposition.
\end{proof}

\begin{proof}[Proof of Theorem \ref{thm_intro:finite_onto_double_cover}] 
Let $Z$ be a $\mathbb{Q}$-Fano variety of dimension $n \ge 3$ with Gorenstein canonical singularities, $Y$ a double cover of $(\p^1)^n$ branched along a reduced divisor $B=\sum_{j\in J}B_j$ of type
$(2,\ldots,2)$, and $f\colon Z \to Y$ a finite morphism. Suppose that 
$B_j$ is ample for each $j\in J$.

Let $p_1\colon  (\p^1)^n \to \p^1$ be any projection, and let 
$p^1\colon  (\p^1)^n \to (\p^1)^{n-1}$ be the projection onto the product of the remaining factors.
Denote by $\pi\colon Y \to (\p^1)^n$ the natural morphism, and set $\pi_1=p_1\circ \pi$ and $\pi^1=p^1\circ \pi$ so that $\pi=(\pi_1,\pi^1)$.
Let $Y_1$ be a general fiber of $\pi_1$. Then $\pi_{|Y_1}$ makes $Y_1$
a double cover of 
$\pi(Y_1)\cong(\p^1)^{n-1}$ branched along the reduced divisor 
$B_{|\pi(Y_1)}$.

We show that there is a dense open subset $U \subset \p^1$ with the following property. For any 
$w\in U$, denote by $Y_w$ the fiber of $\pi_1$ over $w$. Then any irreducible component of
$B_{|\pi(Y_w)}$ is ample. 
To prove the claim, suppose to the contrary that there is  
a dense subset $W \subset \p^1$ such that, for any $w\in W$, some irreducible component of 
$B_{|\pi(Y_w)}$ is not ample. 
Notice that an effective divisor on $\pi(Y_w)$ is not ample
if and only if it is the pullback of some effective divisor under a projection
$\pi(Y_w)\cong(\p^1)^{n-1}\to(\p^1)^{n-2}$. 
After replacing $W$ with a subset of it if necessary, we may assume that there exists a projection $q^1 \colon (\p^1)^{n-1} \to (\p^1)^{n-2}$ such that, for any $w\in W$, 
$B_{|\pi(Y_w)}$ contains in its support the pullback of a prime divisor under
$q^1\circ {p^1}_{|\pi(Y_w)} \colon \pi(Y_w) \to (\p^1)^{n-2}$.
This implies that $B$ contains in its support the pullback of some effective divisor
under $(p_1,q\circ p^1)\colon (\p^1)^n \to \p^1\times(\p^1)^{n-2}$. But this contradicts 
the fact that 
$B_j$ is ample for each $j\in J$, proving our claim.
Thus, we may assume that any irreducible component of $B_{|\pi_1(Y_1)}$
is ample.

By Lemma \ref{lemma:double_cover_connected}, $Z_1=f^{-1}(Y_1)$ is connected. Moreover, $Z_1$ has Gorenstein canonical singularities, and the restriction of $f$ to $Z_1$ induces a finite morphism
$f_1\colon Z_1\to Y_1$ with $\deg(f_1)=\deg(f)$. 
Therefore, to prove the proposition, it suffices to consider the case when $n=3$.
By Lemma \ref{lemma:anticanonical_map}, we may assume without loss of generality that $K_{Y_1}^2 \ge 3$.
Then the result follows from Proposition \ref{proposition:double_cover_ter}.
\end{proof}

\begin{cor}
Let $Z$ be a Fano manifold of dimension $\ge 2$, $Y$ a smooth double cover of $\p^1\times \cdots \times \p^1$ branched along a reduced divisor $B$ of type
$(2,\ldots,2)$, and $f\colon Z \to Y$ a finite morphism. 
Then $\deg(f)=1$.
\end{cor}

\begin{proof}If $\dim Y=2$, then the result follows from  Proposition
\ref{proposition:finite_morphism_degree4}. Suppose from now on that $\dim Y \ge 3$. Since $Y$ is smooth, $\textup{Supp}(B)$ is smooth by \ref{say:double_cover}, and hence irreducible. Thus Theorem \ref{thm_intro:finite_onto_double_cover} applies. 
\end{proof}

\section{Proof of Theorem \ref{thm:main}}

We are now in position to prove our main result. Notice that Theorem \ref{thm:main} is an immediate consequence of 
Lemma \ref{lemma:target_log_Fano} and
Theorem \ref{thm:main_log_version} below.

\begin{thm}\label{thm:main_log_version}
Let $(X,B)$ be a $\mathbb{Q}$-Fano pair 
with locally factorial canonical singularities.
Suppose that
$\Nef(X)=\Pef(X)$ and $\rho(X) = \dim X$. Then   
$X \cong X_1 \times \cdots \times X_m$
where $X_i$ is a double cover of 
$\p^1\times\cdots\times\p^1$ branched along a reduced divisor of type
$(2,\ldots,2)$ for each 
$i \in \{1,\ldots,m\}$.
\end{thm}

\begin{proof}Notice that $B\in\Nef(X)$.
Thus $X$ is a $\mathbb{Q}$-Fano variety with locally factorial canonical singularities (see \cite[Corollary 2.35]{kollar_mori}). 
To prove Theorem \ref{thm:main_log_version}, we argue by induction on $n=\dim X$. By Lemma \ref{lemma:nef_pseff_picard}, the Mori cone
$\NE(X)$ is simplicial, and $\dim \NE(X) = n$ by assumption.

If $n=2$, then the result follows from Proposition \ref{proposition:nef_psef_surfaces}.

Suppose that $n \ge 3$.

Let $V$ be a face of $\NE(X)$, and denote by $\phi\colon X \to Y$ the corresponding contraction. Then $\phi$ is equidimensional and $\dim Y=n - \dim V$ by Lemma \ref{lemma:nef_pseff_picard}. 
From Lemma \ref{lemma:image_nef_egal_psef}, we infer that $\Nef(Y)=\Pef(Y)$.
By Lemma \ref{lemma:target_log_Fano}, there exists an effective $\mathbb{Q}$-divisor $B_Y$ on $Y$ such that 
$(Y,B_Y)$ is $\mathbb{Q}$-Fano with log terminal singularities. Finally,
by Corollary \ref{corollary:factoriality_contraction}, $Y$ is locally factorial. Arguing as above, we conclude that
$Y$ is a $\mathbb{Q}$-Fano variety with canonical singularities. Hence, $Y$ satisfies the conclusion of Theorem \ref{thm:main}
provided that $\dim Y < \dim X$, or equivalently $\{0\}\subsetneq V\subsetneq \NE(X)$ by Lemma \ref{lemma:nef=psef_contractions}.

Suppose first that there is a face $\{0\}\subsetneq V\subsetneq \NE(X)$ such that 
$Y$ is a double cover of $(\p^1)^{m}$
branched along a divisor $C$ of type
$(2,\ldots,2)$ with $\dim Y= m \ge 2$. 
Notice that $C$ is reduced by \ref{say:double_cover} and
irreducible by Lemma \ref{lemma:factoriality_hypersurface_ring}.
If $m=2$, then $Y$ is a Del Pezzo surface of degree $4$ satisfying 
$\Nef(Y)=\Pef(Y)$. But this contradicts
Proposition \ref{proposition:nef_psef_surfaces}, and therefore, we must have $\dim Y = m \ge 3$.
Let $W$ be the face of $\NE(X)$ such that 
$V+W=\NE(X)$ and $\dim W = n- \dim V = m$. 
Denote by $\psi\colon X \to T$ the contraction of $W$, and consider
the finite morphism $f=(\phi,\psi) \colon X \to Y \times T$. 
Let $Z$ be a general fiber of $\psi$. Then $Z$ is a $\mathbb{Q}$-Fano variety with Gorenstein canonical singularities, and the restriction of $f$ to $Z$ induces a finite morphism
$f_{|Z}\colon Z \to Y$ with $\deg(f_{|Z})=\deg(f)$. 
By Theorem \ref{thm_intro:finite_onto_double_cover}, we conclude that 
$\deg(f_{|Z})=1$. The theorem follows easily in this case. 

Suppose now that, for each face $\{0\} \subsetneq V\subsetneq \NE(X)$,
we have $Y \cong (\p^1)^m$ for some $m \ge 1$.
Then apply Proposition \ref{proposition:intro} to conclude
that, either $X \cong (\p^1)^n$, or $X$ is a double cover of 
$(\p^1)^n$ branched along a reduced divisor of type $(2,\ldots,2)$.
This completes the proof of the theorem.
\end{proof}

The following example shows that Theorem \ref{thm:main} does not hold for $\mathbb{Q}$-Fano varieties with
Gorenstein canonical singularities. 

\begin{exmp}\label{exmp:quotient_singularities}
Fix integers $m \ge 2$ and $n \ge 2$, and let $\zeta$ be a primitive $m^{\textup{th}}$ root of unity. 
The group $G=\langle \zeta \rangle$ acts on $\p^1\times\cdots\times\p^1$ by 
$$\zeta\cdot (x_1,y_1)\times\cdots\times (x_n,y_n)=(x_1,\zeta y_1)\times\cdots\times (x_n,\zeta y_n).$$
Set $X=(\p^1\times\cdots\times\p^1)/G$. 
Then $X$ has $\mathbb{Q}$-factorial log terminal singularities and Picard number $\rho(X) \le \dim X =n $ by 
\ref{pullback}, and satisfies $\NE(X)=\Pef(X)$. Notice that $X$ admits a finite morphism onto 
$(\p^1/G)\times\cdots\times(\p^1/G)\cong \p^1\times\cdots\times \p^1$. Thus $\rho(X)\ge\dim X$, and hence $\rho(X)=\dim X$. 
The variety $X$ has isolated singularities and the natural morphism 
$\p^1\times\cdots\times \p^1 \to X$ is \'etale over $X \setminus \Sing(X)$. This implies that 
$X$ is a $\mathbb{Q}$-Fano variety.
We claim that $X$ is not locally factorial. Notice that the open subset 
$U=\{x_1\cdots x_n\neq 0\}\cong \mathbb{A}^n \subset \p^1\times\cdots\times\p^1$
is $G$-stable.
Set $W=U/G$ with natural morphism $\pi\colon U \to W$. We claim that the ideal sheaf of the reduced 
hypersurface $H=\pi(\{y_1=0\})$
is not locally free at $\pi\big((0,\ldots,0)\big)$. Suppose to the contrary that 
$\sI_H$ is invertible at $(0,\ldots,0)$.
The functions $y_1y_2^{m-1}$ and $y_1^m$ on $U$ are $G$-invariant, and hence there exist functions $g_1$ and $g_2$ on $W$ such that 
$y_1y_2^{m-1}=g_1\circ\pi$ and $y_1^m=g_2\circ\pi$. 
Then $\sI_{H,(0,\ldots,0)}=g_1\sO_{W,(0,\ldots,0)}$ and 
$g_2\in \sI_{H,(0,\ldots,0)}^m=g_1^m\sO_{W,(0,\ldots,0)}$
since $y_1y_2^{m-1}=g_1\circ\pi$ and $y_1^m=g_2\circ\pi$ vanish at order 
$1$ and $m$ respectively along $\{y_1=0\}$. This in turn implies
that $y_1^m=g_2\circ\pi$ vanishes at order $\ge m(m-1)$ along $\{y_2=0\}$, yielding a contradiction.
This proves that $X$ is not locally factorial.
Finallly, by \cite{watanabe}, $X$ is Gorenstein if and only if $n-2k\equiv 0 \, [m]$ for any 
$0 \le k \le n$, e.g. $n$ is even and $m=2$.

\medskip

Suppose that $n \ge 3$. Then $X$ does not satisfy the conclusion of Theorem \ref{thm:main}. To prove the claim, we argue by contradiction, and assume that $X \cong X_1 \times \cdots \times X_s$
where $X_i$ is a double cover of 
$\p^1\times\cdots\times\p^1$ branched along a reduced divisor of type
$(2,\ldots,2)$ for each $i \in \{1,\ldots,s\}$. Since $X$ has isolated singularities, we must have $s=1$. 
By Lemma \ref{lemma:pi_1_revetement_double}, we conclude that $\pi_1\big(X\setminus \Sing(X)\big)=\{1\}$, yielding a contradiction
since $\pi_1\big(X\setminus \Sing(X)\big)\cong G$.
\end{exmp}

The following proposition is a first step towards the classification of Fano manifolds $X$ with $\Nef(X)=\Pef(X)$ and arbitrary Picard number.

\begin{prop}\label{proposition:main2}
Let $(X,B)$ be a $\mathbb{Q}$-Fano pair 
with locally factorial canonical singularities satisfying $\Nef(X)=\Pef(X)$. Suppose that there is a contraction $\phi\colon X \to Y$ with $\dim Y = \rho (Y)$ and
$Y\not\cong \p^1\times\cdots\times\p^1$.
Then $X \cong Y_1 \times Y^1$ where $Y_1$ and $Y^1$ are positive dimensional  
$\mathbb{Q}$-Fano varieties with locally factorial canonical singularities,
$\Nef(Y_1)=\Pef(Y_1)$, and $\rho(Y_1)=\dim Y_1$.
\end{prop}

\begin{proof}Notice that $B\in\Nef(X)$. 
Therefore, $X$ is a $\mathbb{Q}$-Fano variety with locally factorial canonical singularities. 
By Lemma \ref{lemma:target_log_Fano}, $\Nef(Y)=\Pef(Y)$, and
there exists an effective $\mathbb{Q}$-divisor $B_Y$ on $Y$ such that 
$(Y,B_Y)$ is $\mathbb{Q}$-Fano with log terminal singularities. From Corollary \ref{corollary:factoriality_contraction}, we deduce that $Y$ is locally factorial. 
Thus Theorem \ref{thm:main_log_version} applies to $(Y,B_Y)$.
We have $Y \cong Y_1 \times \cdots \times Y_m$
where $Y_i$ is a double cover of 
$\p^1\times\cdots\times\p^1$ branched along a reduced divisor $B_i$ of type
$(2,\ldots,2)$ for all
$1 \le i \le m$.
Since $Y\not\cong \p^1\times\cdots\times\p^1$, we may assume without loss of generality that 
$\dim Y_1 \ge 2$. 
From Corollary \ref{corollary:factoriality_contraction} applied to $X \to Y_1$, we deduce that $Y_1$ is locally factorial. 
Therefore, by Proposition \ref{proposition:nef_psef_surfaces}, we must have $\dim Y_1\ge 3$.
Notice that $B_1$ is
irreducible by Lemma \ref{lemma:factoriality_hypersurface_ring}.

By Lemma \ref{lemma:contraction} applied to $\varphi_1\colon X \to Y_1$, there exists a contraction 
$\varphi^1\colon X \to Y^1$ such that the induced morphism 
$\varphi=(\varphi_1,\varphi^1)\colon X \to Y_1\times Y^1$ is surjective and finite.
Let $Z$ be a general fiber of $\varphi^1$. Then $Z$ is a $\mathbb{Q}$-Fano variety with Gorenstein canonical singularities. Moreover, the restriction of $\varphi_1$ to $Z$ induces a finite morphism
$f\colon Z \to Y_1$ with $\deg(f)=\deg(\varphi)$. 
From Theorem \ref{thm_intro:finite_onto_double_cover}, we conclude that
$\deg(\varphi)=1$, completing the proof of the proposition.
\end{proof}

\begin{rem}In the setup of Proposition \ref{proposition:main2}, $Y_1$ satisfies the conclusion of Theorem \ref{thm:main}.
\end{rem}

\section{Finite morphisms between smooth Del Pezzo surfaces}\label{section:finite_morphisms_surfaces}

In order to prove Theorem \ref{thm:main}, we were led to study finite morphisms between $\mathbb{Q}$-Fano varieties. In this section
we address finite morphisms between smooth Del Pezzo surfaces. 
We believe that these results are interesting on their own and that they will be useful when considering 
Fano manifolds with $\Nef(X)=\Pef(X)$ and arbitrary Picard number.

Beauville classified smooth Del Pezzo surfaces which admit an endomorphism of
degree $>1$ in \cite{beauville_endo}. 
We will consider a finite morphism  
$f\colon S \to T$ of degree $>1$ between
smooth Del Pezzo surfaces $S$ and $T$ with $K_S^2\neq K_T^2$, or equivalently $K_S^2<K_T^2$.
We will show that we must have $K_T^2 \ge 8$ (see Theorem \ref{thm:morphism_Del_pezzo_surface}). A smooth Del Pezzo surface of degree 8 
(respectively, 9) is isomorphic to $\p^1\times\p^1$ or to $\p^2$ blown up at one point
(respectively, $\p^2)$). 
Notice that any projective surface admits a finite morphism 
of degree $>1$ onto $\p^2$ by Noether Normalization Lemma. 
The classification of smooth Del Pezzo surfaces which admit a finite morphism of degre $>1$ onto $\p^1\times\p^1$ 
follows easily from
Proposition \ref{proposition:finite_morphism_quadric}.

\begin{say}[Smooth Del Pezzo surfaces]
It is well known that a smooth Del Pezzo surface is either $\p^1\times\p^1$
or the blow up of $\p^2$ at $1 \le r \le 8$ points in general position, namely, no three points on a line, no six points on a conic, and if $r=8$, not all of them on a cubic with one of the point being a singular point.
\end{say}

We will make use of the following elementary observation.

\begin{lemma}\label{lemma:conic_bundle_as_blow_up}
Let $S\not\cong\p^1\times\p^1$ be a smooth Del Pezzo surface, and let $f_1 \colon S \to \p^1$ be a conic bundle structure on $S$. Set $r=9-K_S^2$.
Then there are $r$ points in general position
$p_1,\ldots,p_r$ in $\p^2$ such that the following holds. The surface
$S$ is the blow up of $\p^2$ in 
$p_1,\ldots,p_r$ and
$f_1$ is induced by the pencil of lines in $\p^2$ passing through $p_1$.
\end{lemma}

\begin{proof}Notice that $\rho(S)=10-K_S^2=r+1$.
Let $\mu \colon S \to S_1$ be a minimal model of $f_1 \colon S \to \p^1$, and denote
the exceptional prime curves of $\mu$ by $E_1,\ldots, E_{r-1}$. Then 
$S_1 \cong \mathbb{F}_m$ for some integer $m \ge 0$, and $f_1$ is induced by the natural morphism
$\mathbb{F}_m \to \p^1$. 
Let $\sigma_1 \subset S_1$ be a minimal section, and let $\sigma$ be the strict transform of $\sigma_1$ in $S$. Then 
$\sigma^2 \le \sigma_1^2 = -m$, and hence $m \le 1$ since the only curves with negative square on a smooth Del Pezzo surface are the exceptional ones.

Suppose first that $S_1 \cong \mathbb{F}_1$. Notice that
$\sigma_1 \cap \mu(E_i)=\emptyset$ for all $1 \le i\le r-1$. 
Consider the blow down $\nu \colon S_1 \to \p^2$ of $\sigma_1$, and 
set $\mu_1=\nu\circ\mu \colon S \to \p^2$. Then $S$ is obtained by blowing up  
$\mu_1(E_1),\ldots,\mu_1(E_{r-1})$, and $\nu(\sigma_1)$ in $\p^2$, and  
$f_1$ is induced by the pencil of lines in $\p^2$ passing through $\nu(\sigma_1)$.

Suppose now that $S_1 \cong \p^1\times\p^1$. Since $S \not\cong \p^1\times\p^1$, 
we must have $r\ge 2$. Let $S_0$ be the blow up of $S_1$ at 
$\mu(E_1)$, and let $F_1$, $F_2$ and $F_3$ be the exceptional curves
on $S_0$. Up to renumbering the $F_i$'s if necessary, we may assume that $F_1$ is contracted by the natural morphism $S_0 \to S_1$. 
Let $\nu\colon S_0 \to \p^2$ be the blow down of $F_2$ and $F_3$. Denote by 
$\mu_1\colon S \to \p^2$ the natural morphism.
Then $S$ is obtained by blowing up 
$\mu_1(E_{2}),\ldots,\mu_1(E_{r-1}),\nu(F_2)$, and $\nu(F_3)$ in $\p^2$. Moreover, 
up to renumbering the $F_i$'s if necessary, $f_1$ is induced by the pencil of lines in $\p^2$ passing through $\nu(F_2)$.
\end{proof}

\begin{lemma}\label{lemma:degree2_quadric}
Let $S$ be a smooth Del Pezzo surface of degree $K_S^2 = 2$, and 
let $f_1\colon S \to \p^1$ and $f_2\colon S \to \p^1$ be two conic bundle structures on $S$.
Suppose that $f=(f_1,f_2) \colon S \to \p^1\times\p^1$ is finite.
Then there exist $7$ points $p_1,\ldots,p_7$ in general position in $\p^2$ such that 
$S$ is the blow up of $\p^2$ at $p_1,\ldots,p_7$, and such that
$f_1$ is induced by the pencil of lines in $\p^2$ passing through $p_1$. Moreover, one of the following holds.

\begin{itemize}
\item $f_2$ is induced by the pencil of cubics passing through $p_2,\ldots,p_7$ with a point of multiplicity
$2$ at $p_7$, and $\deg(f)=3$. 
\item $f_2$ is induced by the pencil of quartics passing through $p_1,\ldots,p_7$ with a point of multiplicity $2$ at $p_5,p_6$, and $p_7$, and $\deg(f)=3$. 
\item $f_2$ is induced by the pencil of quintics passing through $p_1,\ldots,p_7$ with a point of multiplicity $2$ at 6 out of $p_1,\ldots,p_7$, and either $\deg(f)=4$ or $\deg(f)=3$ according to whether or not $p_1$ is a multiple point.
\end{itemize}
Conversely, given $S$, $f_1$, and $f_2$ as above, $f=(f_1,f_2)$ is a finite morphism.
\end{lemma}

\begin{proof}By Lemma \ref{lemma:conic_bundle_as_blow_up},
there exist $7$ points $p_1,\ldots,p_7$ in general position in $\p^2$ such that 
$S$ is the blow up of $\p^2$ at $p_1,\ldots,p_7$, and such that
$f_1$ is induced by the pencil of lines in $\p^2$ passing through $p_1$.
The reducible fibers of $f_1$ are the reducible curves $\ell_i\cup F_i$ for $i=2,\ldots,7$,
where $\ell_i$ denotes the strict transform
in $S$ of the line connecting $p_1$ and $p_i$, and $F_1,\ldots, F_7$ denote the exceptional curves in $S$ over
the $p_i$'s.
By \cite[8.7.1, p. 454]{dolgachev_book}, up to renumbering the $p_i$'s if necessary, one of the following holds.
\begin{enumerate}
\item $f_2$ is induced by the pencil of lines passing through $p_2$.
\item $f_2$ is induced by the pencil of conics passing through 4 out of $p_1,\ldots,p_5$.
\item $f_2$ is induced by the pencil of cubics passing through $p_1,\ldots,p_6$ with a point of multiplicity
$2$ at $p_1$ or $p_2$.
\item $f_2$ is induced by the pencil of cubics passing through $p_2,\ldots,p_7$ with a point of multiplicity
$2$ at $p_7$.
\item $f_2$ is induced by the pencil of quartics passing through $p_1,\ldots,p_7$ with a point of multiplicity $2$ at $p_1,p_2$, and $p_3$.
\item $f_2$ is induced by the pencil of quartics passing through $p_1,\ldots,p_7$ with a point of multiplicity $2$ at $p_5,p_6$, and $p_7$.
\item $f_2$ is induced by the pencil of quintics passing through $p_1,\ldots,p_7$ with a point of multiplicity $2$ at 6 out of $p_1,\ldots,p_7$.
\end{enumerate}
In cases (1), (3), and (5), $f_1$ and $f_2$ contract $\ell_2$. This contradicts the fact that $f$ is finite.
In case (2), $f_1$ and $f_2$ contract $F_6$, yielding again a contradiction. Therefore, $f_2$ satisfies 
one of the conditions in the statement of the lemma.

Conversely, let $S$ be the blow-up of $\p^2$ in 
$7$ points $p_1,\ldots,p_7$ in $\p^2$ in general position. 
Let $f_1$ be the conic bundle structure on $S$ induced by 
the pencil of lines in $\p^2$ passing through $p_1$, and let $f_2$ be a conic bundle structure on $S$
satisfying any of the conditions in the statement of the lemma. Set
$f=(f_1,f_2) \colon S \to \p^1\times\p^1$. To prove that $f$ is finite, we have to check that there is no curve on $S$ contracted by both $f_1$ and $f_2$. 
Denote by $C(m,m_{i_1}p_{i_1},\ldots,m_{i_s}p_{i_s})$ the strict transform in $S$ of an integral plane curve of degree $m$ passing through $p_{i_1},\ldots,p_{i_s}$ with multiplicities $m_{i_1},\ldots,m_{i_s}$ respectively, where $\{i_1,\ldots,i_s\}\subset \{1,\ldots,7\}$.

Suppose first that we are in case (4) in the statement of Lemma \ref{lemma:degree2_quadric}.
Then the reducible fibers of $f_2$ are the curves
$C(3,p_1,\ldots,p_6,2p_7)\cup F_1$ and
$C(2,p_i,p_j,p_k,p_l,p_7)\cup C(1,p_m,p_7)$ where 
$\{i,j,k,l,m\}=\{2,3,4,5,6\}$. It follows that $f$ is finite in this case.

Suppose now that we are in case (6) in the statement of Lemma \ref{lemma:degree2_quadric}. 
Then the reducible fibers of $f_2$ are the curves
$C(3,p_1,\ldots,p_4,2p_i,p_j,p_k)\cup C(1,p_j,p_k)$
where $\{i,j,k\}=\{5,6,7\}$, together with the curves
$C(2,p_i,p_j,p_5,p_6,p_7) \cup C(2,p_k,p_l,p_5,p_6,p_7)$ where
$\{i,j,k,l\}=\{1,2,3,4\}$. We conclude again that $f$ is finite.

Finally, suppose that we are in case (7) in the statement of Lemma \ref{lemma:degree2_quadric}. Namely,
$f_2$ corresponds to the pencil of quintics passing through $p_1,\ldots,p_7$ with a point of multiplicity $2$ at $p_{i_1},\ldots,p_{i_6}$ where $\{i_1,\ldots,i_6\}\subset \{1,\ldots,7\}$. 
Then the reducible fibers of $f_2$ are the curves
$C(3,2p_{j_1},p_{j_2},\ldots,p_{j_6},p_7)\cup C(2,p_{j_2},\ldots,p_{j_6},p_7)$
where $\{j_1,\ldots,j_6\}=\{i_1,\ldots,i_6\}$.
We infer that $f$ is finite in this case too, proving the lemma.
\end{proof}

\begin{prop}\label{proposition:finite_morphism_quadric}
Let $S$ be a smooth Del Pezzo surface. Then there exists a finite morphism
$S \to \p^1\times \p^1$ if and only if $K_S^2 \in \{1,2,4\}$, or $K_S^2=8$ and $S\cong \p^1\times \p^1$. 
\end{prop}

\begin{proof} 
Suppose first that there exists a finite morphism $f \colon S \to \p^1\times \p^1$. Denote by $f_1\colon S \to \p^1$ and $f_2\colon S \to \p^1$ the natural projections. By replacing $f_i$ with its Stein factorization, we may assume that it has connected fibers. The result then follows from
Proposition \ref{proposition:finite_morphism_quadric_bis}.

Conversely, let $S$ be a smooth Del Pezzo surface with $K_S^2 \in \{1,2,4\}$.

If $K_S^2=4$, then $S$ is a double cover of $\p^1\times\p^1$ (see proof of Proposition \ref{proposition:finite_morphism_degree4}).

Suppose that $K_S^2 \le 2$. By Lemma \ref{lemma:degree2_quadric}, it suffices to consider the case when 
$K_S^2=1$. Recall that $S$ is obtained by blowing up $8$ points in general position
$p_1,\ldots,p_8$ in $\p^2$. 
Denote by $F_1,\ldots, F_8$ the exceptional curves in $S$ over
the $p_i$'s, and by
$C(m,m_{i_1}p_{i_1},\ldots,m_{i_s}p_{i_s})$ the strict transform in $S$ of an integral plane curve of degree $m$ passing through $p_{i_1},\ldots,p_{i_s}$ with multiplicities $m_{i_1},\ldots,m_{i_s}$ respectively,
where $\{i_1,\ldots,i_s\}\subset \{1,\ldots,8\}$.

Let $f_1$ be the conic bundle structure on $S$ induced by 
the pencil of lines in $\p^2$ passing through $p_1$, and  
let $f_{2}$ be the 
conic bundle structure on $S$ induced by the pencil of
plane quartics through $p_2,\ldots,p_8$ with multiplicity $2$ at $p_6,p_7$ and $p_8$.
Set $f=(f_1,f_2)\colon S\to \p^1\times\p^1$.
To show that $f$ is finite, we have to check that there is no curve on $S$ contracted by both $f_1$ and $f_2$.
The reducible fibers of $f_1$ are the curves $C(1,p_1,p_i)\cup F_i$ where $i \in\{2,\ldots,r\}$.
The reducible fibers of $\pi_2$ are the curves
$C(3,p_2,\ldots,p_5,2p_i,p_j,p_k)\cup C(1,p_j,p_k)$
where $\{i,j,k\}=\{6,7,8\}$, together with the curves
$C(2,p_i,p_j,p_6,p_7,p_8) \cup C(2,p_k,p_l,p_6,p_7,p_8)$ where
$\{i,j,k,l\}=\{2,3,4,5\}$, and $C(4,p_1,\ldots,p_5,2p_6,2p_7,2p_8)\cup F_1$. 
We conclude that $f$ is finite, completing the proof of the proposition.
\end{proof}

\begin{say}\label{finite_morphism_exceptional_curves}Let $S$ and $T$ be smooth Del Pezzo surfaces, and let $f \colon S \to T$ be a finite morphism.
Let $F$ be an exceptional curve on $T$, and consider 
the blow down $\nu\colon  T \to N$ of $F$. Let 
$\mu\colon S\to M$ be the Stein factorization of the composite map 
$S \to T \to N$, and let $g \colon M \to N$ be the induced morphism. 
Because of the genus formula $C^2+C\cdot K_T=2p_a(C)-2$, the only curves with negative square on $T$ are the exceptional ones. By the Hodge Index Theorem, we conclude that 
$f^{-1}(F)$ is the (disjoint) union of exceptional curves.
Thus $M$ and $N$ are smooth Del Pezzo surfaces, and $g$ is a finite morphism with 
$\deg(g)=\deg(f)$.
\end{say}

We now classify endomorphisms of smooth Del Pezzo surfaces of degree 7.

\begin{lemma}\label{endomorphism_S_7}
Let $S$ be the Del Pezzo surface given by blowing up $\p^1\times\p^1$ at 
$p \in \p^1\times\p^1$, and let $f \colon S \to S$ be an endomorphism. Then there exist a positive integer $m$ and a choice of coordinates such that $p=(0,1)\times (0,1)$, and either $f$ is given by 
$(x_1,x_2)\times (y_1,y_2) \mapsto (x_1^m,x_2^m)\times (y_1^m,y_2^m)$, or 
$f$ is given by $(x_1,x_2)\times (y_1,y_2) \mapsto (y_1^m,y_2^m) \times (x_1^m,x_2^m)$.
\end{lemma}

\begin{proof}Denote by $\nu\colon S \to \p^1\times\p^1$ the natural morphism, with exceptional locus $F$.
By \ref{finite_morphism_exceptional_curves} and Proposition \ref{proposition:finite_morphism_quadric}, there is a commutative diagram
$$
\xymatrix{
    S \ar[r]^{f} \ar[d]_{\mu} & S \ar[d]^{\nu} \\
    \p^1\times\p^1 \ar[r]_{g} & \p^1\times \p^1
  }
$$
where $g$ is a finite morphism. There exist finite morphisms 
$g_1 \colon \p^1 \to \p^1$ and $g_2 \colon \p^1 \to \p^1$ such that $g$ is given 
by $q=(q_1,q_2) \mapsto \big(g_1(q_1),g_2(q_2)\big)$, or $q=(q_1,q_2) \mapsto \big(g_2(q_2),g_1(q_1)\big)$.

Recall that there are 3 exceptional curves $F$, $F_1$ and $F_2$ on $S$. Moreover, $F_1\cap F_2=\emptyset$ and $F \cdot F_1=F\cdot F_2=1$. This implies that $f^*F=\delta F$ for some positive integer 
$\delta$. In particular, we have $g^{-1}(p)=\{p\}$. Therefore, there exist positive integers $l$ and $m$ and a choice
of coordinates such that $p=(0,1)\times (0,1)$, and such that 
$g_1$ and $g_2$ are given by $(x_1,x_2)\mapsto (x_1^l,x_2^l)$ and 
$(y_1,y_2)\mapsto (y_1^m,y_2^m)$ respectively.
Thus
$g^{-1}\sI_{p}\cdot \sO_{\p^1\times\p^1,p}=\big((\frac{x_1}{x_2})^l,(\frac{y_1}{y_2})^m\big)\sO_{\p^1\times\p^1,p}$.
But $(g \circ \mu)^{-1}\sI_p\cdot \sO_S\cong \sO_S(-\delta F)$ is a line bundle on $S$. A straightforward computation then shows that we must have $l=m$.
\end{proof}

We are now ready to prove the main result of this section.

\begin{proof}[Proof of Theorem \ref{thm:morphism_Del_pezzo_surface}]
Let $S$ and $T$ be smooth Del Pezzo surfaces with $K_S^2<K_T^2$, and let 
$f \colon S \to T$ be a finite morphism. To prove Theorem \ref{thm:morphism_Del_pezzo_surface},
we argue by contradiction, and assume that $K_T^2 \le 7$.

\medskip

\noindent\textit{Step 1.} Suppose first that $K_T^2=7$. 
Then $T$ is the blow up of $\p^1\times\p^1$ at some point $p=(p_1,p_2) \in \p^1\times\p^1$.
Denote by $\nu\colon S \to \p^1\times\p^1$ the natural morphism, with exceptional locus $F$.
By \ref{finite_morphism_exceptional_curves}, there is a commutative diagram
$$
\xymatrix{
    S \ar[r]^{f} \ar[d]_{\mu} & T \ar[d]^{\nu} \\
    M \ar[r]_{g} & \p^1\times \p^1
  }
$$
where $M$ is a smooth Del Pezzo surface with $K_{M}^2>K_{S}^2$, and $g$ is a finite morphism with
$\deg(g)=\deg(f)$.
Denote by $g_i\colon M \to \p^1$ with $i\in\{1,2\}$ the natural projections, and denote by
$h_i\colon M \to \p^1$ their Stein factorizations. There exist endomorphisms 
$u_i\colon \p^1\to \p^1$ such that $g_i=u_i \circ h_i$.
Set $u=(u_1,u_2)\colon \p^1\times\p^1 \to \p^1\times\p^1$ and $h = (h_1,h_2) \colon M \to \p^1\times\p^1$ so that $g = h \circ u $.

We show that 
\begin{equation}\label{eq:cardinality}
\#\, u_1^{-1}(p_1) 
= \#\, u_2^{-1}(p_2) = \#\, g^{-1}(p) =1
\end{equation}
%$\#\, u_1^{-1}(p_1) 
%= \#\, u_2^{-1}(p_2) = \#\, g^{-1}(p) =1$. 
Suppose that $\#\, u_1^{-1}(p_1) < \#\, g^{-1}(p)$. Notice that 
$g^{-1}(p)\subset g_1^{-1}(p_1)$, and hence 
$h_1(g^{-1}(p))\subset h_1(g_1^{-1}(p_1))=u_1^{-1}(p_1)$. Thus,
there exist $m_1\neq m_2$ on $M$ with $h_1(m_1)=h_1(m_2)$ (and $g(m_1)=g(m_2)=p$). The fiber 
$h_1^{-1}(h_1(m_1))$ is either a smooth connected curve with self intersection zero or the union of two exceptional curves on $M$. Hence, $\mu^{-1}(h_1^{-1}(h_1(m_1)))$ contains a smooth rational curve with 
self intersection $\le -2$, yielding a contradiction. This proves that
$\#\, u_1^{-1}(p_1) \ge \#\, g^{-1}(p)$.
Similarly, we have
$\#\, u_2^{-1}(p_2) \ge \#\, g^{-1}(p)$, and thus
$$\#\, g^{-1}(p) \ge  \#\, u^{-1}(p)=\#\, u_1^{-1}(p_1) \,
 \#\, u_2^{-1}(p_2)
\ge \big(\#\, g^{-1}(p)\big) ^2.$$
This implies that $\#\, u_1^{-1}(p_1) 
= \#\, u_2^{-1}(p_2) = \#\, g^{-1}(p) =1$, proving \eqref{eq:cardinality}. Set $q=g^{-1}(p)$.

By \eqref{eq:cardinality}, we must have $K_{M}^2=K_S^2+1$, and hence
$2 \le K_{M}^2=K_S^2+1<K_T^2+1=8$. 
Thus $K_{M}^2\in\{2,4\}$ by Proposition \ref{proposition:finite_morphism_quadric}.
This implies that
$\deg(h)\ge 2$. 

Since $\#\, u_1^{-1}(p_1) = \#\, u_2^{-1}(p_2) = 1$, there exist positive integers $l$ and $m$ and a choice of coordinates such that $p_1=(0,1)$ and $p_2=(0,1)$, and such that
$u_1$ and $u_2$ are given by $(x_1,x_2)\mapsto (x_1^l,x_2^l)$ and 
$(y_1,y_2)\mapsto (y_1^m,y_2^m)$ respectively.
Denote the exceptional prime curve of $\mu$ by $E$. We have
$f^*F=\delta E$ for some positive integer $\delta$. Taking squares gives $\deg(f)=\delta^2.$
Since $g = h \circ u $, we obtain 
\begin{equation}\label{eq:degree}
\delta^2=\deg(f)=\deg(g)=\deg(h)lm.
\end{equation}
Notice that the multiplicity of $x_1\circ h_1$ 
(respectively, $y_1\circ h_2$)
at $q$ is $\le 2$.
Since $(g\circ\mu)^{-1}\sI_p\cdot\sO_S\cong \sO_S(-\delta F)$, 
a straightforward computation shows that
\begin{equation}\label{eq:multiplicity}
\delta \in \{l,2l\}\cap \{m,2m\}.
\end{equation}
We may assume without loss of generality that $l\le m$. From
\eqref{eq:degree} and \eqref{eq:multiplicity}, we obtain that
either $\delta=m=2l$ and $\deg(h)=2$, or $\frac{\delta}{2}=m=l$ and $\deg(h)=4$.

Suppose first that $K_{M}^2=4$. By Proposition \ref{proposition:finite_morphism_quadric_bis}, we must have 
$\deg(h)=2$, and $\sO_M(-K_M)\cong h^*\big(\sO_{\p^1}(1)\boxtimes\sO_{\p^1}(1)\big)$.
Therefore, $\delta=m=2l$.
A straightforward computation gives 
$\sO_S(-mK_S)\cong f^*\Big(\nu^*\big(\sO_{\p^1}(2)\boxtimes \sO_{\p^1}(1)\big)\otimes\sO_{T}(-F)\Big)$, yielding a contradiction since $\nu^*\big(\sO_{\p^1}(2)\boxtimes \sO_{\p^1}(1)\big)\otimes\sO_{T}(-F)$ is not an ample line bundle.

Suppose now that $K_{M}^2=2$. By Lemma \ref{lemma:degree2_quadric}, we must have 
$\deg(h)\in\{3,4\}$. Thus, $\deg(h)=4$ and $\frac{\delta}{2}=m=l$. Moreover,
$\sO_M(-2K_M)\cong h^*\big(\sO_{\p^1}(1)\boxtimes\sO_{\p^1}(1)\big)$.
A straightforward computation gives 
$\sO_S(-2mK_S)\cong f^*\Big(\nu^*\big(\sO_{\p^1}(1)\boxtimes \sO_{\p^1}(1)\big)\otimes\sO_{T}(-F)\Big)$. Again, this yields a contradiction since $\nu^*\big(\sO_{\p^1}(1)\boxtimes \sO_{\p^1}(1)\big)\otimes\sO_{T}(-F)$ is not an ample line bundle.

\medskip

\noindent\textit{Step 2.} Suppose that $K_T^2 \le 6$.

\medskip

Let $N$ be the Del Pezzo surface with $K_{N}^2=7$ given by blowing up $\p^1\times\p^1$ at 
$(0,0) \in \p^1\times\p^1$, and let $\nu\colon T \to N$ be any birational morphism.
Denote
the exceptional prime curves of $\nu$ by $F_1,\ldots, F_{k}$, and set $p_i=\nu_1(F_i)$.
By \ref{finite_morphism_exceptional_curves}, there is a commutative diagram
$$
\xymatrix{
    S \ar[r]^{f} \ar[d]_{\mu} & T \ar[d]^{\nu} \\
    M \ar[r]_{g} & N
  }
$$
where $M$ is a smooth Del Pezzo surface with $K_{M}^2>K_{S}^2$, and $g$ is a finite morphism with
$\deg(g)=\deg(f)$. By Step 1, we must have 
$M\cong N$. Let $R$ be the ramification divisor of $g$.
By Lemma \ref{endomorphism_S_7}, 
$g(R)=\sum_{i=0}^{4}R_i$ where
$R_0$ is the exceptional curve of $\varepsilon\colon N \to \p^1\times\p^1$, $R_1$ and $R_2$ are the strict transforms in $N$ of the curves $\{0\} \times \p^1$ and $\p^1 \times \{0\}$ respectively, and 
$R_3$ and $R_4$ are the strict transforms in $N$ of the curves $\{\infty\} \times \p^1$ and $\p^1 \times \{\infty\}$ for some point $\infty\neq 0$ in $\p^1$. Moreover, we have
$R=g^{-1}\big(g(R)\big)\cong g(R)$.

Fix $i\in\{1,\ldots,k\}$. We claim that $p_i\in \big(N\setminus \textup{Supp}(R_0+R_1+R_2)\big)\cup\big(R_3\cap R_4)$.
Notice that $R_0,R_1,R_2$ are exceptional curves on $N$, and hence
$p_i\in N \setminus \textup{Supp}(R_0+R_1+R_2)$.
Suppose that $p_i \in R_3 \setminus \textup{Supp}(R_2+R_4)$ and let $q_i \in g^{-1}(p_i)$. 
Then there is a choice of local coordinates
at $q_i$ and $p_i$ such that $q_i=(0,0)$ and
$g$ is given by $(x,y)\mapsto (x^m,y)$ for some integer $m \ge 2$. In particular,
$g^{-1}\sI_{p_i}\cdot\sO_{M,q_i}=(x^m,y)\sO_{M,q_i}$. A straightforward computation shows that
$(g\circ\mu)^{-1}\sI_{p_i}\cdot\sO_S$ is not invertible, yielding a contradiction.
Thus $p_i \not\in R_3 \setminus \textup{Supp}(R_2+R_4)$. Similarly,
$p_i \not\in R_4 \setminus \textup{Supp}(R_1+R_3)$. This proves the claim.
Notice that $\#\, g^{-1}(p_i)=1$ if $p_i \in R_3\cap R_4$, while
$\#\, g^{-1}(p_i)=m^2$ if $p_i\in N\setminus \textup{Supp}(g(R))$.
We have
$$\sum_{1 \le i\le k}\#\, g^{-1}(p_i)= K_{M}^2-K_S^2= 7-K_S^2 \le 6.$$
Observe that $\#\, g^{-1}(p_{i_0}) \ge 2$ for some $1 \le i_0 \le k$ since
$K_S^2<K_T^2$. It follows from the above discussion that $p_{i_0}\in N\setminus \textup{Supp}(g(R))$.
We conclude that $m=2$. Let $C$ be the strict transform in $N$ of a ruling of $\p^1\times \p^1$ passing through $\varepsilon(p_{i_0})$. Then 
$g^{-1}(C)$ is the union of $2$ disjoint smooth rational curves with zero self intersection, and 
$\mu$ blows up 4 points on $\textup{Supp}(g^{-1}(C))$, yielding a contradiction. This completes the proof of the theorem.
\end{proof}

\begin{exmp}

Set $M=\p^1\times\p^1$, $N=\p^2$, and consider the double cover 
$g\colon M\to N$ induced by the projection $\p^3 \dashrightarrow \p^2$ from a general point.
Let $T \simeq \mathbb{F}_1$ be the blow up of $N$ a general point $p$, 
$S$ the blow up of $M$ along $g^{-1}(p)$, and $f\colon S \to T$ the induced (finite) morphism.
If we denote the exceptional curve of $\nu \colon T \to N$ by $F$, then an easy computation gives 
$\sO_S(-K_S) \simeq f^* \big(\nu^*\sO_{\p^2}(2)\otimes \sO_S(-F)\big)$, and therefore, $-K_S$ is ample.
We have $K_S^2=6$ and $K_T^2=8$.
\end{exmp}

\def\cprime{$'$}
\providecommand{\bysame}{\leavevmode\hbox to3em{\hrulefill}\thinspace}
\providecommand{\MR}{\relax\ifhmode\unskip\space\fi MR }
\providecommand{\MRhref}[2]{%
  \href{http://www.ams.org/mathscinet-getitem?mr=#1}{#2}
}
\providecommand{\href}[2]{#2}

%\bibliographystyle{amsalpha}
%\bibliography{bib}

\begin{thebibliography}{OSCWW14}

\bibitem[Amb05]{ambro_lc_trivial}
Florin Ambro, \emph{The moduli {$b$}-divisor of an lc-trivial fibration},
  Compos. Math. \textbf{141} (2005), no.~2, 385--403.

\bibitem[And85]{ando}
Tetsuya Ando, \emph{On extremal rays of the higher-dimensional varieties},
  Invent. Math. \textbf{81} (1985), no.~2, 347--357.

\bibitem[AW97]{andreatta_wisniewski_contractions}
Marco Andreatta and Jaros{\l}aw~A. Wi{\'s}niewski, \emph{A view on contractions
  of higher-dimensional varieties}, Algebraic geometry---{S}anta {C}ruz 1995,
  Proc. Sympos. Pure Math., vol.~62, Amer. Math. Soc., Providence, RI, 1997,
  pp.~153--183.

\bibitem[BCDD03]{bcdd03}
Laurent Bonavero, Cinzia Casagrande, Olivier Debarre, and St{\'e}phane Druel,
  \emph{Sur une conjecture de {M}ukai}, Comment. Math. Helv. \textbf{78}
  (2003), no.~3, 601--626.

\bibitem[Bea01]{beauville_endo}
Arnaud Beauville, \emph{Endomorphisms of hypersurfaces and other manifolds},
  Internat. Math. Res. Notices (2001), no.~1, 53--58.

\bibitem[Cam92]{campana92}
Fr{\'e}d{\'e}ric Campana, \emph{Connexit\'e rationnelle des vari\'et\'es de
  {F}ano}, Ann. Sci. \'Ecole Norm. Sup. (4) \textbf{25} (1992), no.~5,
  539--545.

\bibitem[Cas08]{cinzia_quasi_elementary}
Cinzia Casagrande, \emph{Quasi-elementary contractions of {F}ano manifolds},
  Compos. Math. \textbf{144} (2008), no.~6, 1429--1460.

\bibitem[CD89]{cossec_dolgachev}
Fran{\c c}ois~R. Cossec and Igor~V. Dolgachev, \emph{Enriques surfaces. {I}},
  Progress in Mathematics, vol.~76, Birkh\"auser Boston, Inc., Boston, MA,
  1989.

\bibitem[Dol12]{dolgachev_book}
Igor~V. Dolgachev, \emph{Classical algebraic geometry}, Cambridge University
  Press, Cambridge, 2012, A modern view.

\bibitem[FG12]{fujino_gongyo_adjunction}
Osamu Fujino and Yoshinori Gongyo, \emph{On canonical bundle formulas and
  subadjunctions}, Michigan Math. J. \textbf{61} (2012), no.~2, 255--264.

\bibitem[FS09]{fujino_sato}
Osamu Fujino and Hiroshi Sato, \emph{Smooth projective toric varieties whose
  nontrivial nef line bundles are big}, Proc. Japan Acad. Ser. A Math. Sci.
  \textbf{85} (2009), no.~7, 89--94.

\bibitem[Fuj11]{fujino_mmp}
Osamu Fujino, \emph{Fundamental theorems for the log minimal model program},
  Publ. Res. Inst. Math. Sci. \textbf{47} (2011), no.~3, 727--789.

\bibitem[Ful11]{fulger}
Mihai Fulger, \emph{The cones of effective cycles on projective bundles over
  curves}, Math. Z. \textbf{269} (2011), no.~1-2, 449--459.

\bibitem[GHS03]{ghs03}
T.~Graber, J.~Harris, and J.~Starr, \emph{Families of rationally connected
  varieties}, J. Amer. Math. Soc. \textbf{16} (2003), no.~1, 57--67
  (electronic).

\bibitem[Gro05]{sga2}
Alexander Grothendieck, \emph{Cohomologie locale des faisceaux coh\'erents et
  th\'eor\`emes de {L}efschetz locaux et globaux ({SGA} 2)}, Documents
  Math\'ematiques (Paris) [Mathematical Documents (Paris)], 4, Soci\'et\'e
  Math\'ematique de France, Paris, 2005, S{\'e}minaire de G{\'e}om{\'e}trie
  Alg{\'e}brique du Bois Marie, 1962, Augment{\'e} d'un expos{\'e} de
  Mich{\`e}le Raynaud. [With an expos{\'e} by Mich{\`e}le Raynaud], With a
  preface and edited by Yves Laszlo, Revised reprint of the 1968 French
  original. \MR{2171939 (2006f:14004)}

\bibitem[JR06]{jahnke_radloff}
Priska Jahnke and Ivo Radloff, \emph{Gorenstein {F}ano threefolds with base
  points in the anticanonical system}, Compos. Math. \textbf{142} (2006),
  no.~2, 422--432.

\bibitem[KM98]{kollar_mori}
J{\'a}nos Koll{\'a}r and Shigefumi Mori, \emph{Birational geometry of algebraic
  varieties}, Cambridge Tracts in Mathematics, vol. 134, Cambridge University
  Press, Cambridge, 1998, With the collaboration of C. H. Clemens and A. Corti,
  Translated from the 1998 Japanese original.

\bibitem[KMM87]{kmm}
Yujiro Kawamata, Katsumi Matsuda, and Kenji Matsuki, \emph{Introduction to the
  minimal model problem}, Algebraic geometry, {S}endai, 1985, Adv. Stud. Pure
  Math., vol.~10, North-Holland, Amsterdam, 1987, pp.~283--360.

\bibitem[KMM92]{kmm3}
J.~Koll{\'a}r, Y.~Miyaoka, and S.~Mori, \emph{Rational connectedness and
  boundedness of {F}ano manifolds}, J. Differential Geom. \textbf{36} (1992),
  no.~3, 765--779.

\bibitem[Laz80]{lazarsfeld_barth}
Robert Lazarsfeld, \emph{A {B}arth-type theorem for branched coverings of
  projective space}, Math. Ann. \textbf{249} (1980), no.~2, 153--162.

\bibitem[Li13]{qifeng}
Qifeng Li, \emph{Pseudo-effective and nef cones on spherical varieties},
  Preprint {\tt arXiv:1311.6791}, 2013.

\bibitem[MM82]{mori_mukai}
Shigefumi Mori and Shigeru Mukai, \emph{Classification of {F}ano {$3$}-folds
  with {$B_{2}\geq 2$}}, Manuscripta Math. \textbf{36} (1981/82), no.~2,
  147--162.

\bibitem[Nor83]{nori}
Madhav~V. Nori, \emph{Zariski's conjecture and related problems}, Ann. Sci.
  \'Ecole Norm. Sup. (4) \textbf{16} (1983), no.~2, 305--344.

\bibitem[OSCWW14]{OSCWW}
G.~Occhetta, L.~E. Sol\'a~Conde, K.~Watanabe, and J.~A. Wi{\'s}niewski,
  \emph{Fano manifolds whose elementary contractions are smooth
  $\p^1$-fibrations}, Preprint {\tt arXiv:1407.3658}, 2014.

\bibitem[PCS05]{cheltsov_trigonal}
V.~V. Przhiyalkovski{\u\i}, I.~A. Chel{\cprime}tsov, and K.~A. Shramov,
  \emph{Hyperelliptic and trigonal {F}ano threefolds}, Izv. Ross. Akad. Nauk
  Ser. Mat. \textbf{69} (2005), no.~2, 145--204.

\bibitem[Wat74]{watanabe}
Keiichi Watanabe, \emph{Certain invariant subrings are {G}orenstein. {I},
  {II}}, Osaka J. Math. \textbf{11} (1974), 1--8; ibid. 11 (1974), 379--388.

\bibitem[Wi{\'s}91]{wisn_fano}
Jaros{\l}aw~A. Wi{\'s}niewski, \emph{On contractions of extremal rays of {F}ano
  manifolds}, J. Reine Angew. Math. \textbf{417} (1991), 141--157.

\end{thebibliography}

\end{document}